\documentclass[12pt,letterpaper]{article}
\usepackage[utf8]{inputenc}
\usepackage{amsmath}
\usepackage{amsfonts}
\usepackage{amssymb}
\usepackage{amsthm}
\usepackage{graphicx}
\usepackage[left=1in,right=1in,top=1in,bottom=1in]{geometry}
\usepackage{enumerate}
\usepackage{appendix}
\usepackage{mathtools}

\newtheorem{theorem}{Theorem}[section]
\newtheorem{lemma}[theorem]{Lemma}

\newtheorem{corollary}[theorem]{Corollary}
\theoremstyle{definition}

\theoremstyle{remark}
\newtheorem{remark}[theorem]{Remark}

\DeclareMathOperator{\dv}{div}
\DeclareMathOperator{\curl}{curl}

\DeclareMathOperator{\tr}{tr}
\DeclareMathOperator*{\essup}{ess\,sup}
\DeclareMathOperator*{\essinf}{ess\,inf}

\usepackage{ bbold }

\usepackage[backend=bibtex,style=alphabetic,sorting=nyt]{biblatex}

\numberwithin{equation}{section}

\begin{document}

\title{The Brinkman-Fourier System with Ideal Gas Equilibrium}

\author{Jan-Eric Sulzbach\footnote{Department of Applied Mathematics, Illinois Institute of Technology, Chicago, IL 60616, United States} \and {Chun Liu$^*$}\footnote{corresponding author: cliu124@iit.edu}}

\date{}

\maketitle

\begin{abstract}
In this work, we will introduce a general framework to derive the thermodynamics of a fluid mechanical system, which guarantees the consistence between the energetic variational approaches with the laws of thermodynamics. 
In particular, we will focus on the coupling between the thermal and mechanical forces. 
We follow the framework for a classical gas with ideal gas equilibrium and present the existences of weak solutions to this thermodynamic system coupled with the Brinkman-type equation to govern the velocity field.

\end{abstract}

\section{Introduction}
\label{intro}
The Navier-Stokes-Fourier system, known for its power in modeling the thermodynamics of fluids, is an extension of the classical Navier-Stokes equations.
Although a lot can be modeled in the classical setting, it has its limitations  encompassing of reality. 
More recent versions have included a thermodynamic component in order to better describe more complex system. 
Applications of this can be found in engineering, meteorology and even astrophysics \cite{Bird}, \cite{Holmes}, \cite{Chepyzhov}, \cite{Feireisl7}.
This has opened up avenues for researchers to update and expand upon previous applications of the Navier-Stokes-Fourier system.\\ 

More recent advances and methods for the full Navier-Stokes-Fourier system can be found in \cite{Novotny} and the references therein. Results in the field typically focus on proving existence and uniqueness of solutions. 
Extensive theory on weak solutions to the Navier-Stokes-Fourier system has been developed in \cite{Feireisl5}, \cite{Feireisl1}, \cite{Feireisl3}, and is continuously gaining interest among researchers today. 
The existence of weak solutions can even be extended to other domains, such as when Lipschitz boundaries are present \cite{Poul}. 
Many mathematicians investigate long-term behavior of solutions, as in \cite{Feireisl4}, in order to gain information about equilibrium states, bounds on the energy, etc.\\

For the equations of viscous and heat-conductive gases, it can be shown that unique and global strong solutions exist \cite{Danchin}.
Other results, focusing on different aspects of the thermodynamics of fluids are the study of the gas dynamics of thermal non-equilibrium models \cite{Zeng2}, \cite{Zeng}.
Simplified models with temperature dependent coefficients \cite{DeAnna2} and other thermal effects such as cross diffusion \cite{Bulicek} and heat convection \cite{Nishida} have been studied to see the effect of the temperature dependence in the system of equations.
Other types of non-isothermal models are discussed in \cite{Rocca}, \cite{DeAnna1} and \cite{Pliu}.\\
Most recent results were obtained by \cite{Tarfulea} and \cite{Lai} deriving better a priori bounds and showing positivity for the absolute temperature and by \cite{Sulzbach21}, where the well-posed of a non-isothermal reaction-diffusion sytem in a critical Besov space is shown.
In \cite{Yoshimura} the authors present a different free energy formulation of the Navier-Stokes-Fourier system. 
Other works in this field with similar approaches can be found in \cite{Bulivcek} and \cite{Dreyer}. 
For the Poisson-Nerst-Planck-Fourier system a global existence theorem could be proven \cite{Hsieh}.\\

The new and general approach of this paper is that we start from a given free energy function and dissipation functional.
Next, employing the basic laws of thermodynamic in combination with the energetic variational approach and appropriate constitutive relations, we derive a system of partial differential equations describing the thermodynamics of the fluid.
Due to the generality of this approach it can be easily adapted to different regimes.\\

In this paper we will focus on the ideal gas equilibrium, i.e. the equation of state for the internal energy and the pressure depends linearly on the density and absolute temperature.
For the momentum equation we choose a Brinkman-type diffusion, i.e an interpolation between the Stokes and Darcy's law.
The details can be found in Section 2.

The aim of this paper is to prove existence of local-in-time weak solutions of the  following system of partial differential equations
\begin{align}\label{eq 1.1}
&\partial_t \rho +\dv(\rho u)=0,\\
&\nabla p=\mu \Delta u-\nu \rho u,\\\label{eq 1.3}
&\partial_t s+\dv(s u)=\Delta +\dv\big(\frac{\kappa\nabla \theta}{\theta}\big),
\end{align} 
where $\rho$ is the density, $u$ is the velocity and $\theta$ the absolute temperature. 
The entropy production rate is denoted by $\Delta$ and depends on $\rho$, $\theta$ and $u$.
Moreover, the pressure is defined as $p=k_2\rho\theta$ and corresponds to the ideal gas case and the dissipation part in the momentum equation corresponds to the Brinkman-type equation.
The entropy is defined as $-s=k_2\rho(\log\rho-\frac{3}{2}(\log\theta+1))$.
In addition, we aim at a better understanding of the stability and dependence on the initial data and second law of thermodynamics as expressed in \cite{Dafermos}.

In Section 3 we obtain the weak formulation of this system consisting of equations \eqref{eq 1.1}-\eqref{eq 1.3} and show that the absolute temperature is positive.
Moreover, we derive a priori bounds to the equations.

In Section 4 we state the existence theorem \ref{theorem 2} and the higher regularity theorem \ref{theorem 4}.
In the proofs we follow the ideas and techniques presented in \cite{Feireisl2} and \cite{Novotny}.
Key parts are the application of the div-curl lemma and weak $L^1$-convergence results to control the absolute temperature in the entropy equation.

\section{Derivation of the System}
\label{derivation}

In this section we derive the general model describing the thermal effects of a fluid with the example of the ideal gas case.
For the fluid model we consider a Brinkman-type equation.\\
The unknown variables in the system are:
\begin{enumerate}
\item a non-negative measurable function $\rho=\rho(t,x)$ the mass density;
\item a vector field $u=u(t,x)$ denoting the velocity field of the fluid;
\item a positive measurable function $\theta=\theta(t,x)$ the absolute temperature.
\end{enumerate}
The new approach in the derivation of thermodynamic models focuses on the free energy of the system $\psi(\rho,\theta)$ as a starting point and then applies the laws of thermodynamics and the energetic variational approach to obtain the complete model.

The notion of free energy is a useful concept in classical thermodynamics \cite{Baierlein}, \cite{Salinas} because the change in the free energy is the maximum amount of work that a thermodynamic system can perform in a process at constant temperature, and its sign indicates whether a process is thermodynamically favorable or forbidden.
In the following derivation we use the definition of the Helmholtz free energy.
The second thermodynamic concept we use is the entropy.
Entropy can be understood as the measure of disorder of the system.
Note that for fixed density the entropy is a convex function in the temperature and that for a fixed temperature the entropy is a convex function in the density. 

In the derivation we try to keep the statements as general as possible, since we can apply this framework to different settings, i.e. the porous media case or the Allen-Cahn and Cahn-Hilliard model.\\

From the thermodynamics of the ideal gas, \cite{McQuarrie} and \cite{Berry}, we know that the relation between internal energy and the product of temperature and density is linear and similar for the relation between the pressure and the product of temperature and density.
Working backwards from this observation have the following definitions.

For the ideal gas we have the following definition of the \textbf{free energy}
\begin{align}\label{eq 2.1}
\psi(\rho,\theta)=k_2 \theta\rho\log\rho-k_1\rho\theta\log\theta.
\end{align}
The \textbf{entropy} of a system is defined as follows
\begin{align}\label{eq 2.2}
s(\rho,\theta):= -\psi_\theta= -\rho\big(k_2\log\rho-k_1(\log\theta +1)\big),
\end{align}
where $\psi_\theta$ denotes the derivative of $\psi$ with respect to $\theta$.
\begin{remark}
We note that the $\theta\log\theta$ in the free energy is the weakest convex function with faster than linear growth. 
Thus we have a non-vanishing contribution of the temperature in the entropy.
\end{remark}
Moreover, the convexity of the free energy with respect to the temperature allows us to solve the equation of the entropy \eqref{eq 2.2} for $\theta$.
For the given choice of the free energy this can be done explicitly
\begin{align*}
\theta(\rho,s)= \frac{1}{e}\rho^{k_2/k_1}e^{s/k_1\rho}
\end{align*}
whereas in the general case this can be only done implicitly. \\
The \textbf{internal energy} is the Legendre transform in $\theta$ of the free energy, i.e
\begin{align}\label{eq 2.3}
e(\rho,\theta):= \psi -\psi_\theta \theta=\psi +s\theta= k_1\rho\theta.
\end{align}
\begin{remark}
Instead of having $(\rho,\theta)$ as state variables, we can also consider $(\rho,s)$ as new state variables yielding
\begin{align*}
e_1(\rho,s)=e(\rho,\theta(\rho,s))= \frac{k_1}{e}\rho^{1+k_2/k_1}e^{s/k_1\rho}.
\end{align*}
This interpretation of the free energy is crucial because in the laws of thermodynamics we have $\rho$ and $s$ as state variables.
\end{remark}
We assume that we have a closed system, i.e. $\rho$ satisfies the following \textbf{continuity equation}
\begin{align}\label{eq. 2.4}
\rho_t +\dv(\rho u)=0
\end{align}
Thus integrating over the domain $\Omega$ and assuming that the microscopic velocity $u$ satisfies $u\cdot n=0$ on the boundary we obtain the \textbf{conservation of mass}
\begin{align}\label{eq. 2.5}
\frac{d}{dt}\int_\Omega \rho(t,x)dx =0.
\end{align}
In addition, we assume that $\theta$ moves with the particle and that the kinematics of the temperature $\theta$ are governed by a \textbf{transport equation for the temperature}. 
In other words, the temperature $\theta$ is transported along the trajectory of the flow map with the velocity $u$.
\begin{remark}\label{remark 3}
The assumption of the transport of the temperature $\theta$ with the particles is suitable for situations like the ideal gas, where the temperature is not determined by the environment or the case of a solution, where the temperature is determined by the temperature of the solvent.
\end{remark}
\begin{remark}\label{remark 4}
We note that we have a weak duality of the time evolution of the temperature and the total derivative of the entropy in the following way.\\
If $\theta$ evolves as $\frac{d}{dt}\theta=\theta_t+u\cdot\nabla \theta$ then by testing this equation with the entropy $s$ in the weak form yields
\begin{align*}
\int_\Omega \theta_t s +u\cdot\nabla \theta s\, dx=-\int_\Omega s_t \theta +\dv(su)\theta\, dx.
\end{align*}
Thus the \textbf{kinematics of the entropy} are
\begin{align}\label{eq. 2.6}
\frac{d}{dt}s=s_t+\dv(su)
\end{align}
which are crucial in applying the laws of thermodynamics c.f. equations \eqref{eq 2.15} and \eqref{eq 2.16}.
\end{remark}

Next, we choose the \textbf{total energy} and \textbf{dissipation} as follows
\begin{align*}
E^{tot}=\int_\Omega \psi(\rho,\theta),~~\mathcal{D}^{tot}=\int_\Omega \nu \rho u^2 +\mu |\nabla u|^2
\end{align*}
and employ the energetic variational approach to derive the forces \cite{giga2017}.

\begin{remark}
Note that dissipation depends on both $u$ and $\nabla u$ and hence $\mu$ and $\nu$ can be seen as interpolation parameters between the two pure cases.
\end{remark}

Using the least action principle we have
\begin{align*}
A(x(t))=\int_0^T\mathcal{L}dt=-\int_0^T E^{tot}dt =-\int_0^T\int_\Omega \psi(\rho(x,t),\theta(x,t))dxdt,
\end{align*}
where $A$ denotes the action and $\mathcal{L}$ is the Lagrangian of the action.
Since we have no contribution of the kinetic part in the total energy the variation of the action yields
\begin{align*}
\delta_x A(x)=-\int_0^T\int_\Omega f_{cons}\delta_x dxdt.
\end{align*}
The next step is to compute the variation of the action, where we first rewrite it in Lagrangian coordinates
\begin{align*}
A(x(X,t))=-\int_0^T\int_{\Omega_0^X} \psi\bigg(\frac{\rho_0(X)}{\det F},\theta_0(X)\bigg)\det Fdxdt,
\end{align*}
where $F$ denotes the deformation gradient $F=\frac{\partial x}{\partial X}$.
We set $y(X,t)=\tilde{y}(x(X,t),t)$ and the variation $x^\epsilon=x+\epsilon y$, $F^\epsilon=\frac{\partial x^\epsilon}{\partial X}$.
Then
\begin{align*}
\frac{d}{d\epsilon}\bigg|_{\epsilon=0}A(x^\epsilon(X,t))&=-\frac{d}{d\epsilon}\bigg|_{\epsilon=0}\int_0^T\int_{\Omega_0^X} \psi\bigg(\frac{\rho_0(X)}{\det F^\epsilon}\theta_0(X)\bigg)\det F^\epsilon dXdt\\
&=-\int_0^T\int_{\Omega_0^X}\frac{\partial}{\partial \phi} \psi\bigg(\frac{\rho_0(X)}{\det F},\theta_0(X)\bigg)\times\\
&~~~~~~~\bigg[\frac{-\rho_0}{(\det F)^2}\det F \tr\big(F^{-1}\frac{\partial y}{\partial X}\big)\bigg] \det F dXdt\\
&-\int_0^T\int_{\Omega_0^X} \psi\bigg(\frac{\rho_0(X)}{\det F},\theta_0(X)\bigg)\bigg[\det F \tr\big(F^{-1}\frac{\partial y}{\partial X}\big) F^{-1}\bigg] dXdt.
\intertext{We transform the integral back to Eulerian coordinates and obtain}
&=-\int_0^T\int_{\Omega_t^x} - \frac{\partial}{\partial \rho}\psi\big(\rho,\theta\big) \rho \nabla_x\cdot \tilde{y}+  \psi\big(\rho,\nabla_x \phi ,\theta\big)\nabla_x\cdot\tilde{y} dxdt.
\intertext{Integration by parts yields}
&=-\int_0^T\int_{\Omega_t^x}\bigg[\nabla_x\bigg(\rho \frac{\partial \psi}{\partial \rho}\bigg)-\nabla_x \psi \bigg]\cdot\tilde{y} dxdt\\
&~~~-\int_0^T\int_{\partial\Omega_t^x}\bigg(-\rho \frac{\partial \psi}{\partial \rho}+\psi\bigg)\tilde{y}\cdot n dS_xdt
\intertext{where $\tilde{y}\cdot n=0$ and thus the boundary terms equate to 0. Putting everything together we have}
\delta_x A(x)&=-\int_0^T\int_{\Omega_t^x}\nabla\big(\psi_\rho \rho -\psi\big)\cdot\tilde{y}\, dxdt.
\end{align*}
for an arbitrary smooth vector $\tilde{y}(x,t)$ satisfying $\tilde{y}\cdot n =0$ on $\partial \Omega$.\\
Thus we have
\begin{align*}
f_{cons}= \nabla\big (\rho \psi_\rho-\psi \big)=:\nabla p,
\end{align*}
where we define the gradient of the pressure to equal the conservative force.
By this definition the \textbf{pressure law} satisfies the following relation
\begin{align}\label{eq 2.4}
p(\rho,\theta)=\psi_\rho \rho -\psi= k_2\rho \theta.
\end{align}

\begin{lemma}\label{lemma 0}
The pressure satisfies
\begin{align*}
\nabla p=\rho \nabla \psi_\rho +s\nabla \theta.
\end{align*}
\end{lemma}
\begin{proof}
From the definition of the pressure we have $p(\rho,\theta)=\psi_\rho \rho -\psi$ and thus we compute 
\begin{align*}
\nabla p(\rho,\theta)&=\nabla(\psi_\rho \rho -\psi)=\rho \nabla \psi_\rho +\psi_\rho \nabla \rho-\nabla \psi\\
&=\rho \nabla \psi_\rho +\psi_\rho \nabla \rho-\psi_\rho\nabla \rho-\psi_\theta\nabla \theta= \rho \nabla \psi_\rho+s\nabla \theta.
\end{align*}
\end{proof}

\begin{remark}
From classical thermodynamics and statistical mechanics we know that the \textbf{ideal gas law} is 
\begin{align*}
pV=NkT
\end{align*}
where $V$ is the volume, $N$ the number of particles and $k$ the Boltzmann constant \cite{Salinas}.
Using tools from statistical mechanics, we can derive that 
\begin{align*}
S(N,V,E) \equiv S(N,V^{2/3}E)
\end{align*}
and as a consequence for a reversible adiabatic process we obtain 
\begin{align*}
V^{2/3}E = const.
\end{align*}
This relates the internal energy and the pressure as follows
\begin{align*}
e(\rho,\theta)=\frac{3}{2}p(\rho,\theta).
\end{align*}
Thus we obtain a relation between the two constants $k_1$ and $k_2$
\begin{align*}
k_1=\frac{3}{2}k_2.
\end{align*}
\end{remark}

Following the maximum dissipation law, we now compute the variation of the total dissipation
\begin{align*}
\frac{1}{2}\delta_u \mathcal{D}=\int_\Omega\big(\nu\rho u-\mu\Delta u\big)\tilde{u} dx =\int_\Omega f_{diss}\tilde{u}.
\end{align*}
Using the classical \textbf{Newton's force balance} 
\begin{align*}
f_{cons}+f_{diss}=0
\end{align*}
yields a \textbf{Brinkman-type equation} \cite{Brinkman}, \cite{Durlofsky}, which interpolates between the Darcy's law and the Stokes equation
\begin{align}\label{eq. 2.11}
\nabla p=\mu \Delta u -\nu \rho u.
\end{align}

\begin{remark}
An overview over the energetic variational approach can be found in \cite{Lin}.
More recent approaches in adding the thermodynamics into the framework of the energetic variational approach can be found in \cite{Pliu} and \cite{DeAnna1}.
\end{remark}

The next step is to verify the physicality of the given free energy. 
Thus, combining the above results, we can show that the \textbf{Gibb's equation} is satisfied
\begin{align}\label{eq 2.5}
\theta Ds(\rho,\theta)=De(\rho,\theta)+p(\rho,\theta)D\bigg(\frac{1}{\rho}\bigg),
\end{align}
where we use the notation $Df$ to denote a total differential.
Moreover we can interpret this in terms of classical thermodynamics as follows
\begin{align*}
E=TS-pV~~\text{ and } DE=TDS-pDV.
\end{align*}

For further details in classical thermodynamics we refer to \cite{McQuarrie}, \cite{Berry} and \cite{Salinas}.\\

Next, we provide two useful basic Lemmas.
\begin{lemma}\label{lemma 1}
The internal energy as a function of the density and entropy satisfies
\begin{align*}
{e_1}_s(\rho,s)=\theta
\end{align*}
\end{lemma}
\begin{proof}
Note that $\psi_1(\rho,s)=\psi(\rho,\theta)$. Then, $e_1(\rho,s)=e(\rho,\theta)=\psi +s\theta$ and thus 
\[{e_1}_s={\psi_1}s +\theta +s \theta_s=\psi_\theta \theta_s +\theta+s\theta_s= \theta.\]
\end{proof}
\begin{lemma}\label{lemma 2}
The internal energy in terms of density and entropy is related to the free energy in terms of density and temperature
\begin{align*}
\psi_\rho(\rho,\theta)={e_1}_\rho(\rho,s).
\end{align*}
\end{lemma}
\begin{proof}
We compute the right hand side and obtain 
\begin{align*}
{e_1}_\rho(\rho,s)= \psi_\rho+\psi_\theta \theta_\rho +s \theta_\rho=\psi_\rho .
\end{align*} 
\end{proof}
Now, we state the \textbf{laws of thermodynamic}.
The first law of thermodynamics relates the total derivative of the internal energy with work and heat
\begin{align}\label{eq 2.6}
De_1= \text{ work }+\text{ heat},
\end{align}
where both the work and heat are in divergence form and more specific the heat term can be expressed as $\nabla \cdot q$ with $q$ being the heat flux.\\
We recall that the total change in a quality can be expressed as a divergence term plus some additional term.
Applying this to the entropy we obtain
\begin{align}\label{eq 2.7}
s_t+\dv(su) = \nabla \cdot j +\Delta,
\end{align}
where the kinematics for the entropy are related to the transport of the temperature as expressed in equation \eqref{eq. 2.6}. 
And where we denote $j$ as the entropy flux and $\Delta$ as the entropy production rate.\\

The second law of thermodynamics states that the entropy production is non-negative:
\begin{align}\label{eq 2.8}
\Delta\geq 0.
\end{align}

In order to derive a consistent model we need to supplement our equations by constitutive relations.
The relation between the heat flux $q$ and the entropy flux $j$ is given by \textbf{Durhem relation} 
\begin{align}\label{eq 2.10}
j\theta=q,
\end{align}
and the relation between the heat flux $q$ and the absolute temperature $\theta$ is given by the \textbf{Fourier's law}
\begin{align}\label{eq 2.10 b}
q=\kappa \nabla \theta,
\end{align}
where $\kappa$ is the heat conductivity.\\

In order to find the expression for the entropy production rate $\Delta$ we compute 
\begin{align}\label{eq 2.15}
\frac{d}{dt}\int_\Omega e_1(\rho,s)dx &= \int_\Omega \big[e_{1,\rho}\rho_t +e_{1,s}s_t\big] dx\\
\intertext{Using the kinematics for the density $\rho$ from equation \eqref{eq. 2.4} we obtain }\nonumber
&=\int_\Omega\big[ e_{1,\rho}\big(-\nabla\cdot(\rho u)\big)+e_{1,s}s_t\big] dx\\
\intertext{Applying Lemma \ref{lemma 2} yields}\nonumber
&=\int_\Omega\big[ -\nabla \cdot\big(e_{1,\rho}\rho u\big)+ \big(\nabla \psi_\rho\rho\big)\cdot u +e_{1,s}s_t \big]dx\\
\intertext{In order to have the full expression for the gradient of the pressure we have to incorporate the term $s\nabla\theta$ which can only occur if the kinematics for the entropy are as in equation \eqref{eq. 2.6}. And by equation \eqref{eq 2.7} we have}\nonumber
&=\int_\Omega\big[ -\nabla \cdot\big(e_{1,\rho}\rho u+e_{1,s}su\big)+ \big(\nabla \psi_\rho\rho+s\nabla e_{1,s}\big)\cdot u\\\nonumber
&~~~~~ +e_{1,s}\big(\nabla \cdot j +\Delta\big) \big]dx\\
\intertext{By Lemma \ref{lemma 1} and the Durhem equation \eqref{eq 2.10} we have}\nonumber
&=\int_\Omega\big[ -\nabla \cdot\big((e_{1,\rho}\rho+\theta s)u\big)+ \big(\nabla \psi_\rho\rho+s\nabla \theta\big)\cdot u\\\nonumber
&~~~~~ +\nabla \cdot q -\frac{q\cdot\nabla \theta}{\theta}+\theta\Delta \big]dx\\
\intertext{Now, we can apply Lemma \ref{lemma 0} to obtain}\nonumber
&=\int_\Omega \big[-\nabla \cdot\big(e_{1,\rho}\rho u+\theta su\big)+\nabla \cdot q +\nabla(\psi_\rho \rho-\psi)\cdot u\\\nonumber
&~~~~~-\frac{q\cdot\nabla \theta}{\theta}+\theta\Delta \big]dx\\
\intertext{From the definition of the pressure and the absence of external forces and heat sources we have that}\nonumber
&=\int_\Omega\big[ \nabla p\cdot u -\frac{q\cdot\nabla \theta}{\theta}+\theta\Delta\big] dx
\intertext{where we used that the divergence terms equal to zero under the boundary conditions $u\cdot n=0$ and $\nabla \theta\cdot n=0$. Thus we have}\nonumber
&= \int_\Omega\big[ \big(\mu\Delta u -\nu \rho u \big)\cdot u -\frac{q\cdot\nabla \theta}{\theta}+\theta\Delta dx\big]\\
\intertext{and integration by parts yields}\label{eq 2.16}
&=\int_\Omega \big[-\mu|\nabla u|^2 -\nu \rho u^2 -\frac{q\cdot\nabla \theta}{\theta}+\theta\Delta \big]dx
\end{align}
where we used that by equation \eqref{eq. 2.11} $\nabla p=\mu\Delta u -\nu \rho u$.
Since there are no external forces or heat sources in our system the total internal energy must be conserved  and we obtain that
\begin{align}\label{eq 2.12}
\Delta= \frac{1}{\theta}\bigg(\mu |\nabla u|^2+\nu \rho |u|^2 +  \frac{\kappa|\nabla \theta|^2}{\theta}\bigg).
\end{align}
We note that the second law of thermodynamics $\Delta\geq 0$ is satisfied as long as $\theta>0$.
Thus 
\begin{align}
\frac{d}{dt}\int_\Omega e_1(\rho,s)dx =0
\end{align}
and the total entropy is increasing 
\begin{align}
\frac{d}{dt}\int_\Omega s(\rho,\theta)dx =\int_\Omega s_t+\dv(su)dx=\int_\Omega \dv j +\Delta\geq 0.
\end{align}

Combining the above results we are now able to state our system of equations.\\
The \textbf{model equations} are
\begin{align*}
\partial_t \rho +\dv(\rho u)&=0\\
\partial_t s + \dv(su)&= \dv\bigg(\frac{\kappa\nabla \theta}{\theta}\bigg) +\Delta\\
\nabla p&=\mu \Delta u -\nu \rho u,
\end{align*}
where
\begin{align*}
\Delta= \frac{1}{\theta}\bigg(|\nabla u|^2+\nu \rho |u|^2+ \frac{\kappa|\nabla \theta|^2}{\theta}\bigg)
\end{align*}
together with the state equations
\begin{align*}
-s(\rho,\theta)&= k\rho\big(\log\rho-\frac{3}{2}(\log\theta +1)\big)\\
p(\rho,\theta)&=k\rho\theta.
\end{align*}

\begin{remark}
The generality of this new approach in the derivation of the dynamics of thermodynamic systems extends beyond the simple ideal gas case with applications in the porous media equation or the Cahn-Hilliard equation, where the only adaption to the new model takes place in the free energy.
\end{remark}

\section{Preliminaries}
\label{prelim}
Starting from the derivation in the previous section we develop a thermodynamically consistent mathematical model based on the unknown variables $(\rho,\theta,u)$ satisfying the following properties
\begin{enumerate}
\item the problem admits a local-in-time solution for any initial data of finite energy;
\item the total energy of the system remains constant in the absence of external forces or heat sources;
\item the entropy of the system is increasing, i.e. the system evolves to a state maximizing the entropy;
\item weak solutions coincide with classical solutions provided they are smooth enough.
\end{enumerate}
\subsection{Weak Formulation}
\label{weak form}
Now, let us summarize the weak formulation of the problem derived in the previous chapter.  
Moreover, we also specify the minimal regularity of the solutions required and interpret the weak formulation in terms of partial differential equations provided that all quantities are smooth enough.\\

Let $\Omega\subset \mathbb{R}^n$ be a bounded Lipschitz domain, where $n=2,\, 3$.
Then we have the following equations.\\
\textbf{The continuity equation}
\begin{align}
\intertext{Weak formulation:}\label{eq 3.1}
\int_0^T\int_\Omega \rho\big(\partial_t \phi +u\cdot\nabla \phi\big)\,d x d t&=-\int_\Omega \rho_0\phi(0,\cdot)\,d x,
\intertext{where $\phi\in C_c^1([0,T)\times\overline{\Omega})$.}
\intertext{Minimal regularity of solutions required:}
\rho\geq 0,~~\rho\in L^1((0,T)\times\Omega),~~\rho u& \in L^1((0,T)\times\Omega).
\intertext{Formal interpretation:}
\partial_t \rho +\dv(\rho u)&=0~~\text{in }~ (0,T)\times \Omega,\\
\rho(0,\cdot)&=\rho_0,~~u\cdot n|_{\partial\Omega}=0.
\end{align}
\textbf{Brinkman-type equation}
\begin{align}
\intertext{Weak formulation:}
\int_\Omega k\rho\theta (t) \dv\varphi \,d x&=\int_\Omega \mu\nabla u:\nabla \varphi + \nu \rho u\cdot \varphi \,d x
\intertext{where $\varphi \in C_c^1(\overline{\Omega})$ and $\varphi|_{\partial\Omega}=0$.}
\intertext{Minimal regularity of solutions required:}
\rho\theta\in L^1(\Omega),~~&\nabla u\in L^1(\Omega),~~\rho u \in L^1(\Omega).
\intertext{Formal interpretation:}
k\nabla(\rho \theta)&=\mu\Delta u-\nu \rho u~~ \text{in }~\Omega,\\
u|_{\partial\Omega}&=0.
\end{align}
\textbf{Balance of internal energy}
\begin{align}
\intertext{Weak formulation:}
\int_0^T\int_\Omega \rho \theta(t)\partial_t\psi(t)\,d x d t&= \psi(0) E_0
\intertext{where $\psi \in C_c^1[0,T)$.}
\intertext{Minimal regularity of solutions required:}
\rho\theta\in L^1((0,T)\times&\Omega).
\intertext{Formal interpretation:}
\frac{d}{d t}\int_\Omega \rho\theta \,d x=0~~\text{in } (0,T),~~ \int_\Omega\rho_0\theta_0 \,d x&=E_0.
\end{align}
\textbf{Entropy production}
\begin{align}
\intertext{Weak formulation:}\begin{split}
\int_0^T\int_\Omega s\big(\partial_t + u\cdot \nabla \phi\big)\, d x d t &- \int_0^T\int_\Omega \frac{\kappa \nabla \theta}{\theta}\cdot\nabla \phi \,d x d t\\ 
&+\int_0^T\int_\Omega \sigma \phi \,d x d t =-\int_\Omega s_0\phi(0,\cdot)\,d x,\end{split}
\intertext{where $\phi\in C_c^1([0,T)\times\overline{\Omega})$ and }
&\sigma \geq \frac{1}{\theta}\bigg(\mu|\nabla u|^2+ \nu \rho |u|^2 + \frac{\kappa|\nabla\theta|^2}{\theta}\bigg).
\end{align}
\begin{align}
\intertext{Minimal regularity of solutions required:}
\theta>0 ~\text{ a.a on } (0,T)\times \Omega,~~\theta\in L^q((0,T)\times \Omega),~~\nabla \theta\in L^q((0,T)\times \Omega)&,~~q>1\\
s\in L^1((0,T)\times \Omega),~~\frac{\nabla \theta}{\theta}\in L^1((0,T)\times \Omega)&\\
\frac{|\nabla u|^2}{\theta}\in L^1((0,T)\times \Omega) ,~~\frac{\rho|u|^2}{\theta}\in L^1((0,T)\times \Omega),~~\frac{|\nabla \theta|^2}{\theta^2}\in L^1((0,T)\times \Omega)&. 
\intertext{Formal interpretation:}
\partial_t s+\dv(s u)= \sigma +\dv\bigg(\frac{\kappa\nabla \theta}{\theta}\bigg)~~\text{in } (0,T)\times \Omega&,\\
s(0,\cdot)=s_0,~~\nabla \theta\cdot n|_{\partial\Omega}=0,~~~~~~~~&
\intertext{where}\label{eq 3.19}
-s=k\rho(\log \rho -\frac{3}{2}(\log\theta +1)).~~~~~~~&
\end{align}

\subsection{A Priori Estimates}
\label{a priori}
The first observation we make is that the \textbf{conservation of mass} holds, i.e.
\begin{align}\label{eq 3.20}
\int_\Omega \rho(t,\cdot) \,d x=\int_\Omega \rho_0 \,d x= M_0 ~~\text{for a.a. } t\in (0,T).
\end{align}
The second estimate that follows from the balance of energy is the \textbf{energy estimate}
\begin{align}\label{eq 3.21}
ess\sup_{t\in (0,T)}\int_\Omega \rho\theta \,d x \leq C(\rho_0,\theta_0).
\end{align}
From the second law of thermodynamics we observe that the \textbf{total entropy} of the system $S$ defined as $S(t):=\int_\Omega s(t,x) d x$ is non-decreasing, i.e.
\begin{align}\label{eq 3.22}
\int_\Omega s(t,\cdot) \,d x \geq \int_\Omega s_0  \,d x~~\text{for a.a } t\in (0,T).
\end{align}
In the next step we show the \textbf{positivity of the absolute temperature}.

First, we state a Lemma that for for a finite initial entropy the temperature is positive on a set with positive Lebesgue measure.
To this end, we define two regions in the $(\rho,\theta)$-plane:
\begin{itemize}
\item non-degenerate region: low density and/or sufficiently large temperature
\begin{align*}
\rho\leq \overline{Z}\theta,~~\text{for some } \overline{Z}>0;
\end{align*}
\item degenerate region: high density  and/or very low temperature
\begin{align*}
\rho >\overline{Z}\theta,~~\text{for some } \overline{Z}>0.
\end{align*}
\end{itemize} 
Next, we set
\begin{align*}
s_\infty=\lim_{\theta\to 0}s(\rho,\theta)\geq -\infty~~\text{for any fixed }\rho.
\end{align*}

\begin{lemma}\label{lemma 3}
Let $\Omega\subset \mathbb{R}^n$ with $n=2,\,3$ be a bounded Lipschitz domain.
Assume that the non-negative functions $\rho,\,\theta\in L^1((0,T)\times\Omega)$ satisfy
\begin{align*}
\int_\Omega \rho d x=M_0,~~\int_\Omega s d x>M_0s_\infty + \delta~~\text{for some }\delta>0.
\end{align*}
Then there are $\underline{\theta}>0$ and $V_0>0$ such that
\begin{align*}
\big|\{x\in \Omega|~ \theta(\cdot,x)>\underline{\theta}\}\big|\geq V_0.
\end{align*}
\end{lemma}
The Lemma and the idea of the proof are a slight modification of a result by Feireisl and Novotn\'y in \cite[Lemma 2.2]{Novotny}.
\begin{proof}
Assume there exist sequences $\rho_n,\theta_n$ satisfying the assumptions and such that
\begin{align*}
&\rho_n\to \rho ~~\text{in } L^1(\Omega),~~\int_\Omega \rho d x=M_0,~~\big|\{x\in \Omega| ~\theta_n(x)>\frac{1}{n}\}\big|<\frac{1}{n}.
\end{align*}
Then $\theta_n\to 0$ in $L^1(\Omega)$.
Moreover, by the definition of the entropy we have
\begin{align*}
|s(\rho_n,\theta_n)|\leq \rho_n(1+|\log\rho_n|+|\log \theta_n|).
\end{align*}
In the non-degenerate region we obtain
\begin{align*}
\int_{\{\rho_n\leq \overline{Z}\theta\}} s(\rho_n,\theta_n) d x&\leq c \int _{\{\rho_n\leq \overline{Z}\theta\}} \rho_n(1+|\log \rho_n|+|\log \theta_n|) d x\\
&\leq c(\overline{Z})\int_\Omega \theta_n(1+|\log \rho_n|+|\log \theta_n|) d x\to 0
\end{align*}
Thus,
\begin{align*}
\limsup_{n\to \infty} \int_{\{\rho_n\leq \overline{Z}\theta\}} s(\rho_n,\theta_n) d x\leq 0.
\end{align*}
Moreover, in the degenerate region we have
\begin{align*}
\int_{\{\rho_n>\overline{Z}\theta_n\}}s(\rho_n,\theta_n) d x= \int_{\{Z\theta_n\geq\rho_n>\overline{Z}\theta_n\}}s(\rho_n,\theta_n) d x +\int_{\{\rho_n>Z\theta_n\}}s(\rho_n,\theta_n) d x,
\end{align*}
where
\begin{align*}
 \int_{\{Z\theta_n\geq\rho_n>\overline{Z}\theta_n\}}s(\rho_n,\theta_n) d x \leq S(\overline{Z}) Z\int_\Omega \theta_n d x \to 0.
\end{align*}
Thus we conclude that
\begin{align*}
\liminf_{n\to \infty} \int_{\{\rho_n>Z\theta_n\}}s(\rho_n,\theta_n) d x> M_0s_\infty ~~\text{for any } Z>\overline{Z}.
\end{align*}
However, this leads to a contradiction as
\begin{align*}
\int_{\{\rho_n>Z\theta_n\}}s(\rho_n,\theta_n) d x\leq S(Z)\int_{\{\rho_n>Z\theta_n\}}\rho_n d x \to S(Z)M_0,
\end{align*}
where we used the notation
\begin{align*}
s(\rho,\theta)=\rho S(Z),~~Z=\frac{\rho}{\theta},~~ \lim_{Z\to \infty}S(Z)=s_\infty.
\end{align*}
\end{proof}
Thus we have shown that $\int_V \log\theta \,d x $ is finite.
Next, a version of Poincare's inequality provides the following.

\begin{corollary}[{\cite[Proposition 2.2]{Novotny}}]\label{prop 1}
Let $\Omega\subset \mathbb{R}^n$ with $n=2,\,3$ be a bounded Lipschitz domain. 
Let $V\subset \Omega$ be a measurable set such that $|v|\geq V_0>0$.
Then there exists a constant $c(V_0)$ such that
\begin{align*}
\|v\|_{H^1(\Omega)}\leq c(V_0)\bigg(\|\nabla v\|_{L^2(\Omega)}+\int_V |v|dx\bigg).
\end{align*}
\end{corollary}
Applying this result to $v=\log \theta$ we obtain the following estimate for the temperature.
\begin{align}
\int_0^T\int_\Omega |\log \theta|^2+|\nabla\log \theta|^2 dxdt\leq c(data).
\end{align}
This implies the \textbf{positivity of the absolute temperature} with a possible exception on a set of measure zero.

\section{Existence Theory}
\label{existence theory}
The ultimate goal of the forthcoming chapter is to show existence of weak solutions to the system of equations \eqref{eq 3.1}-\eqref{eq 3.19}.
\begin{theorem}[Local-in-time Existence]\label{theorem 2}

Let $\Omega\subset \mathbb{R}^n$ with $n=2,\,3$ be a bounded domain of class $C^{2,\nu}$, $\nu\in (0,1)$.
Assume that the data satisfies the initial conditions
\begin{align*}
&\rho_0\in L^p(\Omega) ~~\text{for some } p\geq 3,~~\int_\Omega \rho_0dx =M_0>0,\\
&E_0=\int_\Omega\rho_0\theta_0 dx <\infty,~~s(\rho_0,\theta_0)\in L^1(\Omega).
\end{align*}
In addition, let the initial density and temperature be positive, i.e.
$\rho_0(x)>0$ and $\theta_0(x)>0$ for all $x\in \Omega$.
Then there exists a time $T>0$ such that the system of equations for the thermal effects of an ideal gas in the Brinkman model admits a weak solution $(\rho,u,\theta)$ on $(0,T)\times\Omega)$ in the sense specified before, i.e. $(\rho,u,\theta)$ satisfy relations \eqref{eq 3.1}-\eqref{eq 3.19}.
\end{theorem}
The proof of the main result consists of several steps outlined as follows:
\begin{itemize}
\item The continuity equation is regularized with an artificial viscosity term and the entropy production equation is replaced by the balance of internal energy.
\item Approximate solutions are obtained by a fixed point method.
\item Performing the first limit we convert the balance of internal energy to an approximate entropy production equation containing an additional small parameter.
\item We pass to the limit in the regularized continuity equation and finally pass to the limit in the entropy production equation.
\end{itemize}

\begin{remark}\label{remark 9}
When the energy law also contains a kinetic part the Brinkman-type equation with additional inertial term becomes the compressible Navier-Stokes-Fourier system for which the ideal gas case is still open \cite{Novotny}.
\end{remark}

\begin{remark}
In the following subsections we set the parameter $k$ to equal one.
This allows to simplify the computations without changing the main theorem.
\end{remark}
\subsection{Approximate scheme}
\label{approx}
The first step in the proof of Theorem \ref{theorem 2} is to regularize the equations \eqref{eq 3.1}-\eqref{eq 3.19}.\\

The \textbf{continuity equation} is regularized by means of an artificial viscosity term
\begin{align}\label{eq 4.1}
\partial_t\rho + \dv(\rho u)=\epsilon\Delta \rho~~\text{in } (0,T)\times\Omega,
\end{align}
with homogeneous Neumann boundary condition
\begin{align}\label{eq 4.2}
\nabla \rho \cdot n=0 ~~\text{on } (0,T)\times \partial\Omega,
\end{align}
and the initial condition
\begin{align}\label{eq 4.3}
\rho(0,\cdot)=\rho_{0,\delta},
\end{align}
where
\begin{align}\label{eq 4.4}
\rho_{0,\delta}\in C^{2,\nu}(\overline{\Omega}),~~\inf_{x\in \Omega}\rho_{0,\delta}>0,~~\nabla \rho_{0,\delta} \cdot n|_{\partial\Omega}=0.
\end{align}
The \textbf{Brinkman-type equation}
\begin{align}\label{eq 4.5}
\int_\Omega \rho \theta \dv \phi dx= \int_\Omega \mu \nabla u :\nabla \phi dx +\int_\Omega \nu \rho u \phi dx,
\end{align}
for any test function $\phi \in X_n$, where
\begin{align}\label{eq 4.6}
X_n\subset C^{2,\nu}(\Omega)\subset L^2(\Omega)
\end{align}
is a finite dimensional vector space endowed with the Hilbert structure of the Lebesgue space $L^2$ and the functions satisfy
\begin{align}\label{eq 4.7}
\phi =0 ~~\text{on } \partial\Omega~~\text{ no-slip boundary conditions.}
\end{align}
Instead of the entropy production, we consider a modified \textbf{internal energy} equation of the form
\begin{align}\label{eq 4.8}\begin{split}
\partial_t e_\delta+\dv(e_\delta u)-\kappa \Delta \theta &= \mu |\nabla u|^2+ \nu \rho |u|^2 -\rho\theta \dv u+\delta\frac{1}{\theta^2}\\
&+\epsilon\delta(\rho^\Gamma+2)|\nabla \rho|^2-\delta \theta^5~~\text{in } (0,T)\times\Omega,\end{split}
\end{align}
with the Neumann boundary condition
\begin{align}\label{eq 4.9}
\nabla \theta\cdot n=0~~\text{on } (0,T)\times\partial\Omega,
\end{align}
and the initial condition
\begin{align}\label{eq 4.10}
\theta(0,\cdot)=\theta_{0,\delta},
\end{align}
where
\begin{align}\label{eq 4.11}
\theta_{0,\delta}\in H^1(\Omega)\cap L^\infty(\Omega),~~ess\inf_{x\in \Omega} \theta_{0,\delta}>0.
\end{align}
Here
\begin{align}\label{eq 4.12}
e_\delta(\rho,\theta)=\rho\theta.
\end{align}
Moreover, the approximate \textbf{internal energy balance} holds
\begin{align}\label{eq 4.13}
\int_\Omega \rho\theta(t)d x=\int_\Omega \rho_{0,\delta}\theta_{0,\delta} d x+\int_0^t\int_\Omega \frac{\delta}{\theta^2}-\delta\theta^5\,d x d \tau.
\end{align}
for all $t\in [0,T]$.\\
The quantities $\epsilon$ and $\delta$ are small positive parameters, yielding better estimates of the approximate scheme.

\subsection{Solvability of the approximate system}
\label{solv approx}
The second step is to show existence of classical solutions to the approximate system.\\

\begin{theorem}[Global existence for the approximate system]\label{theorem 3}
~\\
Let $\epsilon$, $\delta$ be given positive parameters. Under the hypotheses of Theorem \ref{theorem 2} there exists a $\Gamma_0>0$ such that for all $\Gamma>\Gamma_0$ the approximate system \eqref{eq 4.1}-\eqref{eq 4.13} admits a strong solution $(\rho,u,\theta)$ belonging to the following regularity class
\begin{align*}
&\rho\in C([0,T];C^{2,\nu}(\overline{\Omega})),~~\partial_t\rho \in C([0,T];C^{0,\nu}(\overline{\Omega})),~~\inf_{[0,T]\times\overline{\Omega}}\rho>0,\\
&u\in C^1(X_n),\\
&\theta\in C([0,T];H^2(\Omega))\cap L^\infty((0,T)\times\Omega),~~\partial_t\theta \in L^2((0,T)\times\Omega),~~\essinf_{(0,T)\times\Omega}\theta>0.
\end{align*}
\end{theorem}
The idea of the proof follows standard arguments:
\begin{itemize}
\item Given the velocity field $u$, the approximate continuity equation is solved directly by standard parabolic theory;
\item After solving the continuity equation we determine the temperature $\theta$ of the quasilinear parabolic problem, i.e. the internal energy equation, where $u,~\rho$ play the role of given data;
\item To close the loop, the solution $u$ is looked for as the fixed point of an integro-differential operator. 
\end{itemize}
\begin{lemma}[{\cite[Lemma 3.1]{Novotny}}]\label{lemma 4}
Let $\Omega\subset\mathbb{R}^n$ where $n=2,\,3$ be a bounded domain of class $C^{2,\nu},~\nu\in (0,1)$ and let $u\in X_n$ be a given vector field.
Suppose the initial data has the regularity specified in Section 4.1.\\
Then the continuity equation posses a unique classical solution $\rho=\rho_u$, more specifically
\begin{align}\label{eq 4.14}
\rho_u\in V\equiv \big\{\rho\in C([0,T];C^{2,\nu}(\overline{\Omega})),~\partial_t\rho \in C([0,T];C^{0,\nu}(\overline{\Omega}))\big\}.
\end{align}
Moreover, the mapping $u\in X_n\to \rho_u$ maps bounded sets in $X_n$ into bounded sets in $V$ and is continuous with values in $C^1([0,T]\times \overline{\Omega})$.
Finally,
\begin{align}\label{eq 4.15}
\underline{\rho_0}\exp\bigg(-\int_0^\tau \|\dv u\|_{L^\infty(\Omega)} dt\bigg)\leq \rho_u(\tau,x)\leq \overline{\rho_0}\exp\bigg(+\int_0^\tau \|\dv u\|_{L^\infty(\Omega)} dt\bigg),
\end{align}
for all $\tau \in [0,T],~x\in \Omega$, where $\underline{\rho_0}=\inf_{x\in \Omega} \rho_{0,\delta},~\overline{\rho_0}=\sup_{x\in \Omega} \rho_{0,\delta}$.
\end{lemma}
In this part we focus on the quasilinear parabolic problem for the unknown temperature $\theta$.
First, we state a comparison principle.
\begin{lemma}[{\cite[Lemma 3.2]{Novotny}}]\label{lemma 5}
Given the quantities $u,\,\rho_u$ satisfying the regularity $u\in X_n$, $\rho\in C([0,T];C^2(\overline{\Omega}))$, $\partial_t\rho \in C([0,T]\times\overline{\Omega})$, where $\inf_{(0,T)\times\Omega}\rho>0$, and assume that $\underline{\theta}$ and $\overline{\theta}$ are a sub- and super-solution to problem belonging to the regularity class
\begin{align}\label{eq 4.16}
&\underline{\theta},\,\overline{\theta}\in L^2(0,T;H^2(\Omega)),~\partial_t \underline{\theta},\, \partial_t\overline{\theta}\in L^2((0,T)\times\Omega),\\\label{eq 4.17}
 &0<ess\inf_{(0,T)\times\Omega} \underline{\theta}\leq ess\sup_{(0,T)\times\Omega} \underline{\theta}<\infty,~~ 0<ess\inf_{(0,T)\times\Omega} \overline{\theta}\leq ess\sup_{(0,T)\times\Omega} \overline{\theta}<\infty,
\end{align}
and satisfying 
\begin{align}\label{eq 4.18}
\underline{\theta}(0,\cdot)\leq \overline{\theta}(0,\cdot)~~\text{a.e. in } \Omega.
\end{align}
Then 
\begin{align*}
\underline{\theta}(t,x)\leq \overline{\theta}(t,x) ~~\text{a.e. in } (0,T)\times\Omega.
\end{align*}
\end{lemma}
\begin{remark}
If we assume in addition that 
\begin{align}\label{eq 4.19}
0<\underline{\theta}_0=ess\inf_{\Omega}\theta_{0,\delta}\leq ess\sup_{\Omega}\theta_{0,\delta}=\overline{\theta}_0<\infty,
\end{align}
the problem \eqref{eq 4.8}-\eqref{eq 4.12} admits at most one solution in the regularity class specified above.
\end{remark}

\begin{corollary} \label{cor 1}
Let $u,\, \rho_u$ be in the regularity class as before and let the initial data $\theta_{0,\delta}$ satisfy \eqref{eq 4.19}.
Suppose that $\theta$ is a strong solution of the problem belonging to the regularity class.

Then there exists two constants $\underline{\theta}$, $\overline{\theta}$ satisfying $0<\underline{\theta}<\underline{\theta}_0<\overline{\theta_0}<\overline{\theta}$ and
\begin{align}\label{eq 4.20}
\underline{\theta}\leq \theta(t,x)\leq \overline{\theta}~~\text{for a.a. } (t,x)\in (0,T)\times\Omega.
\end{align}
\end{corollary}
\begin{proof}
We check that the constant function $\underline{\theta}$ is a subsolution of the problem
\begin{align*}
\underline{\theta}\partial_t\rho +\underline{\theta}\dv(\rho u)=\mu|\nabla u|^2 +\nu \rho |u|^2-\underline{\theta}\rho\dv u +\delta\frac{1}{\theta^2}+\epsilon\delta(\rho^\Gamma+2)|\nabla \rho|^2-\delta\theta^5.
\end{align*}
This is true as long as 
\begin{align*}
\frac{\delta}{\theta^2}\leq \underline{\theta}\rho\dv u +\underline{\theta}(\partial_t \rho +\dv(\rho u))-\mu|\nabla u|^2-\nu\rho|u|^2-\epsilon\delta(\rho^\Gamma+2)|\nabla \rho|^2.
\end{align*}
We note that all terms on the right-hand-side are bounded in terms of $\|u\|_{X_n}$ and $\|\rho\|_{C^1}$ provided $0<\underline{\theta}<1$.
Then by the comparison principle (Lemma 4.4) the first inequality follows.\\
The upper bound can be established in a similar way by help of the dominating term $-\delta\theta^5$ in equation \eqref{eq 4.8}.
\end{proof}
In addition, we observe the importance of the term $\frac{\delta}{\theta^2}$. This term guarantees that the absolute temperature stays bounded away from zero.

\begin{lemma}[{\cite[Lemma 3.3]{Novotny}}]\label{lemma 6}
Let the data $\rho_u,\,u$ belong to the regularity class as specified above and let the initial data $\theta_{0,\delta}\in H^1(\Omega)$.

Then any strong solution $\theta$ of the problem belonging to the regularity class \eqref{eq 4.16} satisfies the estimate
\begin{align}\label{eq 4.21}\begin{split}
ess\sup_{t\in(0,T)}\|\theta\|_{H^1(\Omega)}^2 &+\int_0^T\big( \|\partial_t\theta\|^2_{L^2(\Omega)}+\|\Delta\theta\|^2_{L^2(\Omega)}\big)dt\\
&\leq C\big(\Omega,\|\rho\|_{C^1([0,t]\times\overline{\Omega})},\|u\|_{X_n}, \inf_{(0,t)\times\Omega}\rho, \|\theta_{0,\delta}\|_{H^1(\Omega)}\big).\end{split}
\end{align}
\end{lemma}

After establishing a priori bounds on the temperature $\theta$ we are able to show existence of strong solutions to the approximate internal energy equation.
The key to this is that those bounds lead to the compactness of the solutions in $L^2(0,T);H^1(\Omega))$. 
Note that we can rewrite the approximate internal energy equation as a quasilinear parabolic equation in the temperature $\theta$. 
For smooth enough data we can apply the results by Ladyzhenskaya \cite{Ladyzhenskaya} to obtain a unique strong solution.

\begin{lemma}\label{lemma 7}
Under the previous assumptions the problem \eqref{eq 4.8}-\eqref{eq 4.12} has a unique strong solution $\theta=\theta_u$  belonging to the regularity class
\begin{align}\label{eq 4.22}\begin{split}
Y=\big\{&\partial_t\theta\in L^2((0,T)\times\Omega),~~\theta\in L^\infty(0,T;H^2(\Omega)\cap L^\infty(\Omega)),\\
&\frac{1}{\theta}\in L^\infty((0,T)\times\Omega)\big\}.\end{split}
\end{align}
Moreover, the mapping $u\to \theta_u$ maps bounded sets in $X_n$ into bounded sets in $Y$ and is continuous with values in $L^2(0,T;H^1(\Omega))$.
\end{lemma}

Now, we are able to show existence of solutions to the approximate system.
We recall that $u\to (\rho_u,\theta_u)$ maps bounded sets in $X_n$ into bounded sets in $C([0,T],C^{2,\nu}(\overline{\Omega}))\times L^\infty(0,T;H^2(\Omega))$ and is continuous with values in $C^1([0,T]\times\Omega)\times L^2(0,T;H^1(\Omega))$.
Moreover from the Brinkman-type equation we obtain using the positivity of $\rho$ that
\begin{align*}
\|\rho u^2\|_{L^1(\Omega)}+\|\nabla u\|^2_{L^2(\Omega)}\leq C\|\rho \theta\|^2_{L^2(\Omega)}\leq C\big(\|\rho\|_{L^4(\Omega)}^2+\|\theta\|_{L^4(\Omega)}^2\big).
\end{align*}
Applying the Poincare inequality and the Sobolev imbedding yields
\begin{align}\label{eq 4.23}
\|u\|_{H^1(\Omega)}\leq C\big(\|\rho\|_{H^1(\Omega)}^2+ \|\theta\|_{H^1(\Omega)}^2\big).
\end{align}
Moreover, we have
\begin{align}\label{eq 4.24}
\|\Delta u\|_{L^2(\Omega)}\leq \|\nabla (\rho\theta)\|_{L^2(\Omega)}\leq \|\rho\|_{L^\infty(\Omega)}\|\nabla \theta\|_{L^2(\Omega)}+\|\theta\|_{L^2(\Omega)}\|\nabla \rho\|_{L^\infty(\Omega)}.
\end{align}
Thus, for each $t\in (0,T)$ $(\rho,\theta)\to u$ maps bounded sets in $C^{2,\nu}(\overline{\Omega})\times H^2(\Omega)$ into bounded sets in $X_n$ and is continuous with values in $H^1(\Omega)$.

Now, a direct application of the Leray-Schauder fixed point theorem yields the existence of a solution $(\rho,u,\theta)$ of the approximate system on a (possibly short) time interval $(0,T_n)$.
Iterating this procedure yields the existence of solutions on $(0,T)$ as long as the bounds are independent of the time $T_n$.

\subsection{Uniform estimates}
\label{uniform est}
In this section we establish uniform bounds, similar to those in chapter 3. 
The existence of such uniform bounds guarantees the global existence of the approximate solution in the desired spaces. 
Moreover, these estimates play a crucial role in the limit passage in the following sections.\\

First, from the approximate continuity equation it follows that the total mass of the system still remains constant in time, i.e.
\begin{align}\label{eq 4.25}
\int_\Omega \rho(t) dx= \int_\Omega \rho_{0,\delta}dx =M_{0,\delta}~~\text{for all } t\in [0,T].
\end{align}
Next, taking $u$ as a test function in the Brinkman-type equation \eqref{eq 4.5} we obtain
\begin{align}\label{eq 4.26}
\int_\Omega \rho\theta\dv u-|\nabla u|^2-\rho |u|^2 dx =0.
\end{align}
From the approximate internal energy equation \eqref{eq 4.8} we have
\begin{align}\label{eq 4.27}
\int_\Omega \rho\theta(t) dx=\int_\Omega \rho_{0,\delta}\theta_{0,\delta} dx +\int_0^T\int_\Omega \frac{\delta}{\theta^2}-\epsilon\delta\theta^5 dx dt.
\end{align} 
Instead of working with the internal energy balance we manipulate the equation \eqref{eq 4.8} to get an approximate entropy production.
To this end, we recall from the classical theory of thermodynamics that
\begin{align*}
\frac{de}{dt}=W+Q~~\text{and } T\frac{ds}{dt}=T\Delta+Q.
\end{align*}
Therefore, we compute
\begin{align*}
&~~~-\partial_t\big(\rho(\log \rho-\log\theta-1)\big)-\dv \big(\rho(\log \rho-\log\theta-1)u\big)\\
&=-\rho_t(\log \rho-\log\theta-1)-\rho_t+\frac{\rho\theta_t}{\theta}-(\log \rho-\log\theta-1)\dv(\rho u)\\
&~~~\,-\rho u\cdot\nabla(\log \rho-\log\theta-1)\\
&=(\log \theta-\log\rho)(\rho_t +\dv(\rho u))+\frac{\rho\theta_t}{\theta}+\dv(\rho u)-u\cdot\nabla \rho +\frac{\rho}{\theta}u\cdot\nabla\theta\\
&=(\log \theta-\log\rho-1)\epsilon\Delta \rho+\rho \dv u+\frac{\rho}{\theta}(\theta_t+u\cdot\nabla \theta)+\epsilon\Delta\rho\\
&=(\log\theta-\log\rho -1)\epsilon\Delta \rho+\rho \dv u+\frac{1}{\theta}\big((\rho\theta)_t+\dv(\rho\theta u)\big)
\end{align*} 
which gives the relation between the internal energy and the entropy of the system.

Then, dividing the internal energy balance by $\theta$ yields 
\begin{align*}
\partial_t s+\dv(su)-\frac{\Delta\theta}{\theta}&=\frac{1}{\theta}(\mu|\nabla u|^2+\nu \rho|u|^2)+ \frac{\delta}{\theta^3}-\delta\theta^4\\
&+\epsilon\Delta\rho(\log\theta-\log\rho -1).
\intertext{Rewriting this expression yields the approximate entropy equation}
\partial_t s+\dv(su)-\dv(\frac{\nabla \theta}{\theta})&=\frac{1}{\theta}(\mu|\nabla u|^2+\nu \rho|u|^2+\frac{|\nabla \theta|^2}{\theta})+\frac{\delta}{\theta^3}-\delta\theta^4\\
&+\epsilon\Delta\rho(\log\theta-\log\rho -1).
\end{align*}
For higher regularity we modify the entropy production rate slightly and obtain
\begin{align}\label{eq 4.28}\begin{split}
\partial_t s+\dv(su)-\dv(\frac{\nabla \theta}{\theta})&=\frac{1}{\theta}\big(\mu|\nabla u|^2+\nu \rho|u|^2+(\frac{\kappa}{\theta}+\delta\theta^{\Gamma-1})|\nabla \theta|^2\big)+\frac{\delta}{\theta^3}\\
&-\delta\theta^4+\epsilon\delta(\Gamma\rho^{\Gamma-2}+2)|\nabla \rho|^2+\epsilon\Delta\rho(\log\theta-\log\rho -1).\end{split}
\end{align}
The two approximate equations give rise to the following estimates, where we multiplied equation \eqref{eq 4.28} by an arbitrary positive constant $\overline{\theta}$ and integrated over the space time domain.
\begin{align*}
&\int_\Omega H_{\delta,\overline{\theta}}(\rho,\theta)(\tau)dx+\overline{\theta}\int_0^\tau\int_\Omega \frac{1}{\theta}\big( \mu |\nabla u|^2+\nu \rho|u|^2+(\frac{\kappa}{\theta}+\delta\theta^{\Gamma-1})|\nabla \theta|^2\big) dxdt\\
 & +\int_0^\tau\int_\Omega \frac{\delta}{\theta^3}+\delta\theta^5 +\epsilon\delta(\Gamma\rho^{\Gamma-2}+2)|\nabla \rho|^2 dxdt \\
 &= \int_\Omega H_{\delta,\overline{\theta}}(\rho,\theta)(0)dx +\int_0^\tau\int_\Omega\frac{\delta}{\theta^2} +\delta\theta^4-\epsilon\overline{\theta}\Delta\rho(\log\theta-\log\rho-1)dxdt,
\end{align*}
where $H_{\delta,\overline{\theta}}(\rho,\theta)=e-s=\rho\theta+\rho(\log\rho-\log\theta-1)$.
Integration by pats in the last term yields
\begin{align}\label{eq 4.29}\begin{split}
&\int_\Omega H_{\delta,\overline{\theta}}(\tau)dx+\overline{\theta}\int_0^\tau\int_\Omega \frac{1}{\theta}(\mu|\nabla u|^2+\nu \rho|u|^2+(\frac{\kappa}{\theta}+\delta\theta^{\Gamma-1})|\nabla \theta|^2\big)+\frac{\delta}{\theta^3} +\delta\theta^5 dxdt\\
& + \int_0^\tau\int_\Omega\epsilon\delta(\Gamma\rho^{\Gamma-2}+2)|\nabla \rho|^2+\epsilon\frac{|\nabla \rho|^2}{\rho} dxdt \\
&=\int_\Omega H_{\delta,\overline{\theta}}(0)dx +\int_0^\tau\int_\Omega \frac{\delta}{\theta^2} +\delta\theta^4+ \epsilon\overline{\theta}\frac{\nabla \rho\cdot\nabla \theta}{\theta}dxdt~~\text{for all }\tau \in [0,T].\end{split}
\end{align}
We observe that we can control the terms on the right-hand side by the terms on the left-hand side and the initial data.

Indeed, the quantity $\frac{\delta}{\theta^2}$ is dominated for low temperatures by its counterpart $\frac{\delta}{\theta^3}$.
Moreover, we estimate
\begin{align*}
\int_0^\tau\int_\Omega \epsilon\overline{\theta}\frac{\nabla \rho\cdot\nabla \theta}{\theta}dxdt &\leq \int_0^\tau\int_\Omega \bigg|\epsilon\overline{\theta}\frac{\nabla \rho\cdot\nabla \theta}{\theta}\bigg|dxdt \\
&\leq \int_0^\tau\int_\Omega\epsilon\overline{\theta}\frac{|\nabla \theta|^2}{2\theta^2}dxdt+\int_0^\tau\int_\Omega\epsilon\overline{\theta}|\nabla\rho|^2 dxdt ,
\end{align*}
where we can bound both terms with terms on the left-hand side of equation \eqref{eq 4.29} for sufficiently small $\epsilon$.

Now, we can summarize the estimates as follows
\begin{align}\label{eq 4.30}
ess\sup_{t\in(0,T)}\int_\Omega H_{\delta,\overline{\theta}}(\rho,\theta)(t)dx\leq c\\\label{eq 4.31}
\int_0^T\int_\Omega \frac{1}{\theta}\bigg(\mu|\nabla u|^2+\nu \rho|u|^2+(\frac{\kappa}{\theta}+\delta\theta^{\Gamma-1})|\nabla \theta|^2\bigg)+\frac{\delta}{\theta^3}+\delta\theta^5 dxdt\leq c\\\label{eq 4.32}
\int_0^T\int_\Omega \epsilon\frac{|\nabla \rho|^2}{\rho}+\epsilon\delta(\Gamma\rho^{\Gamma-2}+2)|\nabla \rho|^2 dxdt\leq c,
\end{align}
where c is a positive constant depending on the initial data but is independent of  $T,\,\epsilon,\,\delta$ and $n$.

\subsection{Limit Passage}
\label{limit}
The third step is to pass to the limit in the approximate/ regularized system.
This is done in three steps. 
First, we let $n \to \infty$. This is the approximation step in the Fourier series for the Brinkman-type equation. 
Next, we pass to the limit in $\epsilon$ as $\epsilon \to 0$, i.e. the additional regularity term for the density $\rho$.
And finally, we let $\delta \to 0$. 
This is the most crucial step because only here we will have a smallness condition on the time interval in which the weak solution exists.

Each step utilizes similar ideas: from finding uniform bounds to the div-curl lemma.
For the details of the theorems used in these steps we refer to the appendix.
  
\subsubsection{Limit $n\to\infty$}
\label{n limit}
Let the vector space $X$ be defined as 
\begin{align*}
X\equiv \bigcup_{n} X_n ~~\text{is dense in } H^1_0(\Omega).
\end{align*}
For $\epsilon>0$ and $\delta>0$ fixed let $(\rho_n,u_n,\theta_n)_{n}$ be a sequence of approximate solutions.

In addition to the uniform estimates \eqref{eq 4.30}-\eqref{eq 4.32} we obtain the following estimates:\\
From the quasilinear parabolic equation for the absolute temperature $\theta$ we obtain that $0\leq \underline{\theta}\leq \theta_n(t,x)\leq \overline{\theta}$ for all $(t,x)\in(0,T)\times\Omega$. 
Hence, $\|\nabla u_n\|^2_{L^2(\Omega)}\leq c$ and together with the boundary condition $ u|_{\partial\Omega}=0$ we have $u_n\in H^1_0(\Omega)$.
Then, by the Sobolev embedding $u_n\in L^6(\Omega)$.
The boundedness of the sequence then implies that
\begin{align}\label{eq 4.33}
u_n\to u~~\text{weakly in } H^1(\Omega).
\end{align}
From the kinetic equation it follows that
\begin{align*}
\frac{d}{dt}\frac{1}{2}\|\rho_n\|_{L^2(\Omega)}^2+\epsilon\|\nabla \rho_n\|^2_{L^2(\Omega)}\leq C\|u_n\|_{L^6(\Omega)}^2\|\rho_n\|^2_{L^4(\Omega)}.
\end{align*}
By the interpolation inequality we have
\begin{align*}
\frac{d}{dt}\|\rho_n\|_{L^2(\Omega)}^2+\epsilon\|\nabla \rho_n\|^2_{L^2(\Omega)}\leq C(M_{0,\delta},\epsilon,\|u\|_{L^6(\Omega)},\Omega).
\end{align*}
Thus the sequence $\rho_n$ is bounded in $L^\infty(0,T;L^2(\Omega))\cap L^2(0,T;H^1(\Omega))$.
Therefore we can assume that
\begin{align}\label{eq 4.34}
\rho_n\to \rho~~\text{weakly-(*) in } L^\infty(0,T;L^2(\Omega)).
\end{align}
By applying Poincare inequality to the additional term in equation \eqref{eq 4.32} we get
\begin{align}\label{eq 4.35}
\rho_n,~\rho_n^{\Gamma/2}~~\text{bounded in } L^2(0,T;H^1(\Omega)),
\end{align}
and by interpolation
\begin{align}\label{eq 4.36}
\rho_n~~\text{bounded in } L^\Gamma(0,T;L^{3\Gamma}(\Omega)).
\end{align}

By the boundedness of the entropy production rate we obtain that
\begin{align}\label{eq 4.37}
&\nabla \theta_n^{\Gamma/2}~~\text{bounded in } L^2(0,T;L^2(\Omega))
\intertext{and}
&\theta_n~~\text{bounded in } L^5((0,T)\times \Omega).
\end{align}
In addition, $\nabla\log \theta_n$ is bounded in $L^2(0,T;L^2(\Omega))$.
This implies that $\theta_n,~\theta_n^{\Gamma/2}\in L^2(0,T;H^1(\Omega))$ and we obtain
\begin{align}\label{eq 4.38}
\theta_n\to \theta~~\text{weakly in } L^2(0,T;H^1(\Omega)).
\end{align}
Moreover, we have
\begin{align}\label{eq 4.39}
\int_0^T\int_\Omega \frac{1}{\theta^3}dxdt\leq \liminf_{n\to \infty} \int_0^T\int_\Omega \frac{1}{\theta_n^3}dxdt.
\end{align}
By the standard Sobolev embedding we derive the higher integrability estimates of $\theta_n$, i.e.
\begin{align}\label{eq 4.40}
\nabla \theta_n~~\text{bounded in } L^\Gamma(0,T;L^{3\Gamma}(\Omega)).
\end{align}
As a byproduct, we get that
\begin{align}\label{eq 4.41}
\log \theta_n~~\text{bounded in } L^q((0,T)\times\Omega)~\text{for any finite } q\geq 1.
\end{align}

Now, we are able to pass to the limit in the equations with respect to the parameter $n$.\\
The limit in the Poisson equation is obtained via the standard Galerkin argument, where we note that
\begin{align}\label{eq 4.42}
\int_\Omega \rho_n\theta_n\dv u_n dx\leq \frac{1}{2}\|\nabla u\|^2_{L^2(\Omega)}+C\big( \|\rho_n\|^2_{L^4(\Omega)}+ \|\theta_n\|^2_{L^4(\Omega)}\big).
\end{align}
Thus we can pass to the limit in the Brinkman equation by the bounds established before
\begin{align}\label{eq 4.43}
-\rho u+ \Delta u=\nabla(\rho \theta) ~~\text{for a.e. } x\in \Omega.
\end{align}
From the kinetic equation we have
\begin{align*}
\big(\partial_t-\epsilon\Delta\big)[\rho_n]=-\nabla\rho_n\cdot u-\rho_n\dv u_n,
\end{align*}
where the terms on the right-hand side are bounded in $L^2(0,T;L^{3/2}(\Omega))$.
Thus
\begin{align}\label{eq 4.44}
\rho_n\to \rho~~\text{ a.e. in } (0,T)\times\Omega
\end{align}
and we can let $n\to \infty$ in the approximate continuity equation to obtain
\begin{align}\label{eq 4.45}
\partial_t\rho +\dv(\rho u)=\epsilon\Delta\rho~~\text{a.e. in } (0,T)\times\Omega,
\end{align}
where the density $\rho$ is a positive function satisfying
\begin{align}\label{eq 4.46}
\nabla \rho(t,\cdot) \cdot n|_{\partial\Omega}=0~~\text{for a.a. } t\in (0,T)
\end{align}
in the sense of traces, together with the initial data
\begin{align}\label{eq 4.47}
\rho(0,\cdot)=\rho_{0,\delta}.
\end{align}
Moreover, 
\begin{align}\label{eq 4.48}
\nabla \rho_n\to \nabla \rho ~~\text{in } L^2(0,T;L^2(\Omega)),
\end{align}
where we used that
\begin{align*}
\int_\Omega \rho_n^2(\tau) dx+2\epsilon \int_0^\tau \int_\Omega |\nabla \rho_n|^2dxdt \to 
\int_\Omega \rho_{\delta,0}^2 dx- \int_0^\tau \int_\Omega \rho^2\dv u dxdt
\end{align*}
and
\begin{align*}
 \int_0^\tau \int_\Omega \rho^2\dv u dxdt = \int_\Omega \rho^2(\tau) dx+2\epsilon \int_0^\tau \int_\Omega |\nabla \rho|^2dxdt.
\end{align*}

Now, we can consider the limit in the entropy balance equation.
The main difficulty here is to show the strong convergence of the temperature.
To this end, we apply the Div-Curl Lemma discovered by Tartar \cite{Tartar} to the function $U$ and $V$, specified below. 
The details off the div-curl Lemma can be found in the Appendix.

We rewrite the equation \eqref{eq 4.28} in the following form
\begin{align}\label{eq 4.50}\begin{split}
\partial_t s&+\dv(su)-\dv(\frac{\nabla \theta}{\theta})-\epsilon\dv(\nabla \rho(\log\theta-\log\rho -1)) \\
&=\frac{1}{\theta}\big(\mu|\nabla u|^2+\nu \rho|u|^2+(\frac{\kappa}{\theta}+\delta\theta^{\Gamma-1})|\nabla \theta|^2\big)+\frac{\delta}{\theta^3}+\delta\theta^5\\
 &+\epsilon\delta(\Gamma\rho^{\Gamma-2}+2)|\nabla \rho|^2+\epsilon\big(\frac{|\nabla \rho|^2}{\rho}+\frac{\nabla\rho \cdot\nabla \theta}{\theta}\big),\end{split}
\end{align}
where we used that
\begin{align*}
\epsilon\Delta\rho(\log\theta-\log\rho -1)=\epsilon\dv(\nabla \rho(\log\theta-\log\rho -1))-\epsilon\big(\frac{|\nabla \rho|^2}{\rho}+\frac{\nabla\rho \cdot\nabla \theta}{\theta}\big).
\end{align*}
Setting
\begin{align*}
U=[s,su-\frac{\nabla \theta}{\theta}-\epsilon\nabla \rho(\log\theta-\log\rho -1)]~~\text{and }~V=[\theta,0,0,0]
\end{align*}
we can check the assumptions for the Div-Curl Lemma.

The temperature $\theta$ is bounded in $L^2((0,T)\times\Omega)$ and $curl(V)$ yields only spatial partial derivatives and thus is bounded in $L^2((0,T)\times\Omega)$ which is compact embedded into $W^{-1,2}((0,T)\times\Omega)$.
By the uniform estimates obtained before we note that the right-hand side of equation \eqref{eq 4.50} is bounded in $L^1((0,T)\times\Omega)$ and therefore precompact in $W^{-1,s}((0,T)\times\Omega)$ provided $s\in [1,\frac{4}{3})$.
Thus it remains to show that $U$ is bounded in a better space than $L^1$.
To see this we note that
\begin{align*}
|s(\rho_n,\theta_n)|\leq c(\rho_n+\rho_n|\log\rho_n|+\rho_n|\log\theta_n|)
\end{align*}
and by the uniform estimates $s$ is bounded in $L^{\Gamma/3}((0,T)\times\Omega)$.
In addition, $su$ is bounded in $L^p((0,T)\times\Omega)$, where $\frac{1}{p}=\frac{1}{2}+\frac{3}{\Gamma}$ provided $\Gamma>6$.

For the other terms we have that $\nabla \log\theta_n$ is bounded in $L^{2}((0,T)\times\Omega)$ and 
\begin{align*}
\epsilon\nabla \rho(\log\theta-\log\rho -1)~~\text{is bounded in } L^{\frac{2\Gamma}{\Gamma+6}}((0,T)\times\Omega).
\end{align*}
Then the Div-Curl Lemma states that
\begin{align}
\overline{s(\rho,\theta)\,\theta}=\overline{s(\rho,\theta)}\theta,
\end{align}
where the symbol $\overline{F(u)}$ denotes the weak $L^1$-limit of the sequence $F(u_n)$ of composed functions.

The goal is now, to conclude that we have almost everywhere convergence of $\theta$. 
By the definition of the entropy we have 
\begin{align*}
s(\rho,\theta)=-\rho(\log \rho-\log\theta-1),
\end{align*}
where we note that the entropy $s$ is increasing in the temperature.
Thus,
\begin{align}
\overline{\rho\log(\theta)\theta}\geq \overline{\rho\log\theta}\theta.
\end{align}
Moreover, by the strong convergence of $\rho_n$ see equation \eqref{eq 4.44}, we have
\begin{align*}
\overline{\rho\log(\theta)\theta}=\rho \overline{(\log\theta)\theta}.
\end{align*}
Combining the above equations we infer that
\begin{align*}
\overline{\log(\theta)\theta}= \overline{\log\theta}\theta.
\end{align*}
By the strict convexity of the function $x\log x$ we have that
\begin{align}
\theta_n\to \theta~~\text{a.e in } (0,T)\times\Omega.
\end{align}
For details in the argument we refer to Appendix.

Now, we can take the limit in the approximate entropy equation.
To this end, we first turn the equation into an inequality by applying Youngs inequality to the $\nabla \rho\cdot \nabla\theta$ term.
\begin{align}
\begin{split}
&\partial_t s+\dv(su)-\dv(\frac{\nabla \theta}{\theta})-\epsilon\dv(\nabla \rho(\log\theta-\log\rho -1)) \\
\geq &\frac{1}{\theta}\big(\mu|\nabla u|^2+\nu \rho|u|^2+(\frac{\kappa}{\theta}+\frac{\delta}{2}\theta^{\Gamma-1})|\nabla \theta|^2\big)+\frac{\delta}{\theta^3}+\delta\theta^5\\
&+\epsilon\delta(\Gamma\rho^{\Gamma-2}+2)|\nabla \rho|^2+\epsilon \frac{|\nabla \rho|^2}{\rho},\end{split}
\end{align}
As a consequence from the previous results we can identify the limits of the individual terms.
\begin{align*}
s(\rho_n,\theta_n)&\to s(\rho,\theta)~~\text{in } L^2((0,T)\times\Omega),\\
s(\rho_n,\theta_n)u_n&\to s(\rho,\theta)u~~\text{weakly in } L^1((0,T)\times\Omega).\\
\end{align*}
The almost all convergence of $\theta_n$ implies that
\begin{align*}
\frac{1}{\theta_n}&\to \frac{1}{\theta} ~~\text{in } L^2((0,T)\times\Omega),\\
\nabla \log\theta_n&\to \nabla \log \theta~~\text{weakly in } L^1((0,T)\times\Omega),\\
\theta_n^{\Gamma-1}\nabla\theta_n&\to \theta^{\Gamma-1}\nabla \theta~~\text{weakly in } L^p((0,T)\times\Omega) ~p>1.
\end{align*}
To control the $\epsilon$-term we note that
\begin{align*}
|\epsilon(\log\theta_n-\log\rho_n-1)\nabla\rho_n|\leq c|\nabla \rho_n|(|\log\theta_n|+|\log\rho_n|+1)
\end{align*}
and all terms on the right are bounded in $L^p((0,T)\times\Omega)$ for some $p>1$.
Thus we have 
\begin{align*}
\epsilon(\log\theta_n-\log\rho_n-1)\nabla\rho_n\to \epsilon(\log\theta-\log\rho-1)\nabla\rho~~\text{weakly in } L^1((0,T)\times\Omega).
\end{align*}
Identifying the limit in the remaining terms of the entropy production rate yields
\begin{align*}
\frac{\nabla u_n}{\sqrt{\theta_n}}&\to \frac{\nabla u}{\sqrt{\theta}}~~ \text{weakly in } L^2((0,T)\times\Omega),\\
\frac{\nabla\rho_n}{\sqrt{\rho_n}}&\to \frac{\nabla \rho}{\sqrt{\rho}}~~ \text{weakly in } L^2((0,T)\times\Omega),\\
\sqrt{\Gamma\rho_n^{\Gamma-2}+2}\nabla \rho_n&\to \sqrt{\Gamma\rho^{\Gamma-2}+2}\nabla \rho~~\text{weakly in } L^2((0,T)\times\Omega).
\end{align*}
These convergence results are sufficient to perform the weak limit in the approximate entropy equation as $n\to \infty$.
We note that the inequality is preserved under the weak limit due to the lower semi-continuity of convex superposition operators.
This allows us to conclude that
\begin{align}\begin{split}
&\int_0^T\int_\Omega s(\rho,\theta)\big(\partial_t\phi+u\cdot\nabla \phi\big)dxdt+\int_0^T\int_\Omega\big(\frac{\nabla \theta}{\theta}-\epsilon(\log \theta-\log\rho-1)\nabla \rho\big)\cdot \nabla \phi dxdt\\
&+\int_0^T\int_\Omega \bigg(\frac{1}{\theta}\big(\mu|\nabla u|^2+\nu \rho|u|^2+(\frac{\kappa}{\theta}+\frac{\delta}{2}\theta^{\Gamma-1})|\nabla \theta|^2\big)+\frac{\delta}{\theta^3}+\delta\theta^5 \bigg)\phi dxdt\\
&+\int_0^T\int_\Omega \bigg(\epsilon\delta(\Gamma\rho^{\Gamma-2}+2)|\nabla \rho|^2+\epsilon \frac{|\nabla \rho|^2}{\rho}\bigg)\phi dxdt \leq -\int_\Omega s(\rho_{\delta,0},\theta_{\delta,0})\phi(0,\cdot)dx\end{split}
\end{align} 
for all $\phi\in C_c^\infty([0,T)\times\overline{\Omega})$.

Thus we can conclude that after the first limit the quantities satisfy the following system of equations.\\

The \textbf{approximate continuity equation}
\begin{align}\label{eq 4.56}
\partial_t\rho + \dv(\rho u)=\epsilon\Delta \rho~~\text{in } (0,T)\times\Omega,
\end{align}
with homogeneous Neumann boundary condition
\begin{align}
\nabla \rho \cdot n=0 ~~\text{on } (0,T)\times \partial\Omega,
\end{align}
and the initial condition
\begin{align}
\rho(0,\cdot)=\rho_{0,\delta}.
\end{align}

The \textbf{Brinkman-type equation}
\begin{align}\label{eq 4.59}
\int_\Omega \rho \theta \dv \phi dx= \int_\Omega \mu \nabla u :\nabla \phi +\nu \rho u\cdot \phi dx,
\end{align}
for any test function $\phi \in C_c^\infty(\overline{\Omega})$ with
\begin{align}
\phi =0 ~~\text{on } \partial\Omega~~\text{ no-slip boundary conditions.}
\end{align}

The \textbf{approximate internal energy balance}
\begin{align}
\int_\Omega \rho\theta(t)d x=\int_\Omega \rho_{0,\delta}\theta_{0,\delta} d x+\int_0^t\int_\Omega \frac{\delta}{\theta^2}-\delta\theta^5\,d x d \tau.
\end{align}
for all $t\in [0,T]$.

The \textbf{approximate entropy inequality}
\begin{align}\begin{split}
&\int_0^T\int_\Omega s(\rho,\theta)\big(\partial_t\phi+u\cdot\nabla \phi\big)dxdt+\int_0^T\int_\Omega\big(\frac{\nabla \theta}{\theta}-\epsilon(\log \theta-\log\rho-1)\nabla \rho\big)\cdot \nabla \phi dxdt\\
&+\int_0^T\int_\Omega \bigg(\frac{1}{\theta}\big(\mu|\nabla u|^2+\nu \rho|u|^2+(\frac{\kappa}{\theta}+\frac{\delta}{2}\theta^{\Gamma-1})|\nabla \theta|^2\big)+\frac{\delta}{\theta^3}+\delta\theta^5 \bigg)\phi dxdt\\
&+\int_0^T\int_\Omega\bigg( \epsilon\delta(\Gamma\rho^{\Gamma-2}+2)|\nabla \rho|^2+\epsilon \frac{|\nabla \rho|^2}{\rho}\bigg)\phi dxdt \leq -\int_\Omega s(\rho_{\delta,0},\theta_{\delta,0})\phi(0,\cdot)dx\end{split}
\end{align} 
for all $\phi\in C_c^\infty([0,T)\times\overline{\Omega})$ with $\phi \geq 0$.\\

Rewriting the last equation
\begin{align*}\begin{split}
&\int_\Omega s(\rho_{\delta,0},\theta_{\delta,0})\phi(0,\cdot)dx-\int_0^T\int_\Omega s(\rho,\theta)\big(\partial_t\phi+u\cdot\nabla \phi\big)dxdt\\
&+\int_0^T\int_\Omega\big(\frac{\nabla \theta}{\theta}-\epsilon(\log \theta-\log\rho-1)\nabla \rho\big)\cdot \nabla \phi dxdt\\
&\geq\int_0^T\int_\Omega \bigg(\frac{1}{\theta}\big(\mu|\nabla u|^2+\nu \rho|u|^2+(\frac{\kappa}{\theta}+\frac{\delta}{2}\theta^{\Gamma-1})|\nabla \theta|^2\big)+\frac{\delta}{\theta^3}+\delta\theta^5\bigg)\phi dxdt\\
&+\int_0^T\int_\Omega \bigg(\epsilon\delta(\Gamma\rho^{\Gamma-2}+2)|\nabla \rho|^2+\epsilon \frac{|\nabla \rho|^2}{\rho}\bigg)\phi dxdt\end{split}
\end{align*} 
for all $\phi\in C_c^\infty([0,T)\times\overline{\Omega})$ with $\phi \geq 0$ we note that the left-hand side of the equation can be understood as a non-negative linear form on the space of smooth function with compact support in $[0,T)\times\overline{\Omega}$.
By the Riesz representation theorem, there exists a regular, non-negative Borel measure $\Sigma_{\epsilon,\delta}$ on $[0,T)\times\overline{\Omega}$ that can be extended to $[0,T]\times\overline{\Omega}$ such that 
\begin{align}\label{eq 4.63}\begin{split}
&\int_0^T\int_\Omega s(\rho,\theta)\big(\partial_t\phi+u\cdot\nabla \phi\big)dxdt+\int_0^T\int_\Omega\big(\frac{\nabla \theta}{\theta}-\epsilon(\log \theta-\log\rho-1)\nabla \rho\big)\cdot \nabla \phi dxdt\\
&+\langle\Sigma_{\epsilon,\delta},\phi \rangle_{[\mathcal{M},C]([0,T]\times\overline{\Omega})} = -\int_\Omega s(\rho_{\delta,0},\theta_{\delta,0})\phi(0,\cdot)dx\end{split}
\end{align} 
for all $\phi\in C_c^\infty([0,T)\times\overline{\Omega})$ with $\phi \geq 0$.
Moreover,
\begin{align*}
\Sigma_{\epsilon,\delta}\geq& \frac{1}{\theta}\big(\mu|\nabla u|^2+\nu \rho|u|^2+(\frac{\kappa}{\theta}+\frac{\delta}{2}\theta^{\Gamma-1})|\nabla \theta|^2\big)+\frac{\delta}{\theta^3} +\delta\theta^5\\
&+\epsilon\delta(\Gamma\rho^{\Gamma-2}+2)|\nabla \rho|^2+\epsilon \frac{|\nabla \rho|^2}{\rho}.
\end{align*}

\subsubsection{Limit $\epsilon\to 0$}
\label{e limit}
The next step is to let $\epsilon\to 0$ in the approximate system.
To this end, let $(\rho_\epsilon, u_\epsilon, \theta_\epsilon)$ be a solution of equations \eqref{eq 4.56}-\eqref{eq 4.63}.\\
 
Then, similar to the previous part we obtain the following estimates independent of $\epsilon$:
\begin{align}\label{eq 4.64}
&\sup_{\epsilon>0}\bigg\{\essup_{t\in (0,T)} \int_\Omega H_{\delta,\overline{\theta}}(t)\,dx\bigg\}<\infty \\
&\sup_{\epsilon>0}\bigg\{\Sigma_{\epsilon,\delta}\big[[0,T]\times \overline{\Omega}\big]\bigg\}<\infty.
\end{align}
This implies that 
\begin{align}
&\sup_{\epsilon>0}\bigg\{\int_0^T\int_\Omega \frac{1}{\theta}\big(\mu|\nabla u|^2+\nu \rho|u|^2+ (\frac{\kappa}{\theta}+\delta\theta^{\Gamma-1})|\nabla \theta|^2\big)+\frac{\delta}{\theta^3}+\delta\theta^5\, dxdt\bigg\}<\infty,\\
&\sup_{\epsilon>0}\bigg\{\epsilon\delta\int_0^T\int_\Omega \big(\Gamma\rho^{\Gamma-2}+2\big)|\nabla \rho|^2\, dxdt\bigg\}<\infty,\\\label{eq 4.68}
&\sup_{\epsilon>0}\bigg\{\epsilon\int_0^T\int_\Omega \frac{|\nabla \rho|^2}{\rho}\,dxdt\bigg\}<\infty.
\end{align}
As in the previous part we conclude that
\begin{align}
&\rho_\epsilon^{\Gamma/2} ~~\text{bounded in } L^2(0,T;H^1(\Omega)),\\
&\theta_\epsilon^{\Gamma/2}~~\text{bounded in } L^2(0,T;H^1(\Omega)).
\end{align}

Thus from the Brinkman equation we obtain
\begin{align}
u_\epsilon ~~\text{bounded in } H^1(\Omega),
\end{align}
and 
\begin{align}
u_\epsilon\to u ~~\text{weakly in } H^1(\Omega).
\end{align}
Passing to the limit in the Brinkman equation yields
\begin{align}\label{eq 4.73}
\int_\Omega \mu \nabla u :\nabla \varphi dx+\int_\Omega \nu\rho u\cdot \varphi dx-\int_\Omega \rho \theta \dv\varphi dx =0,
\end{align}
for any $\varphi \in C_c^\infty(\overline{\Omega})$ with $\varphi|_{\partial\Omega}=0$.\\

Multiplying the approximate continuity equation by $\rho_\epsilon$ and integrating by parts yields
\begin{align*}
\frac{1}{2}\int_\Omega\rho_\epsilon^2(t)dx + \int_0^t\int_\Omega|\nabla \rho_\epsilon|^2dxdt =\frac{1}{2}\int_\Omega\rho_{0,\delta}^2dx -\frac{1}{2}\int_0^t\int_\Omega\rho_\epsilon^2\dv u_\epsilon dxdt.
\end{align*}
Thus, we observe that 
\begin{align}
&\sqrt{\epsilon}\nabla \rho_\epsilon~~\text{is bounded in } L^2(0,T;L^2(\Omega)),\\
&\rho_\epsilon~~ \text{is bounded in } L^\infty(0,T;L^2(\Omega)).
\end{align}
and in particular
\begin{align}
\epsilon\nabla \rho_\epsilon \to 0 ~~\text{in } L^2(0,T;L^2(\Omega)).
\end{align}
This yields that
\begin{align}
\rho_\epsilon u_\epsilon\to \rho u ~~\text{weakly in } L^2(0,T;L^2(\Omega)).
\end{align}
Thus, we can pass to the limit in the approximate continuity equation as $\epsilon\to 0$ and the limit $\rho$ satisfies the integral identity
\begin{align}
\int_0^T\int_\Omega\big(\rho \partial_t\phi + \rho u\cdot \nabla \phi\big) dxdt+\int_\Omega \rho_{0,\delta}dx =0
\end{align}
for any test function $\phi \in C_c^\infty([0,T)\times\overline{\Omega})$ i.e. $\rho,\, u$ satisfy in the following equation in the sense of distributions
\begin{align}
\partial_t\rho +\dv(\rho u)=0.
\end{align}

It remains to pass to the limit in the approximate entropy equation.
To this end, we once more need to show strong convergence of the absolute temperature.
The idea is again to show uniform estimates for $\theta_\epsilon$ and then apply the Div-Curl Lemma.

Taking a closer look at the bounds of the entropy production rate, we derive that
\begin{align*}
&\theta_\epsilon^{\Gamma/2}~~ \text{is bounded in } L^2(0,T;H^1(\Omega)),\\
&\theta_\epsilon^{-1}~~ \text{is bounded in } L^3((0,T)\times\Omega),\\
&\theta_\epsilon~~\text{is bounded in } L^5((0,T)\times	\Omega),\\
&\log\theta_\epsilon ~~ \text{is bounded in } L^2(0,T:H^1(\Omega))\cap L^q((0,T)\times\Omega).
\end{align*}

Now, for the application of the Div-Curl Lemma we use the same idea as in the previous limit case.

Setting
\begin{align}
U_\epsilon=\bigg[s(\rho_\epsilon,\theta_\epsilon),s(\rho_\epsilon,\theta_\epsilon) u_\epsilon + \frac{\nabla \theta_\epsilon}{\theta_\epsilon}+ \epsilon\big(\log\theta_\epsilon-\log\rho_\epsilon-1\big)\nabla\rho_\epsilon\bigg],~~V_\epsilon=[\theta_\epsilon,0,0,0]
\end{align}
we observe that
\begin{align*}
\dv U_\epsilon=\Sigma_{\epsilon,\delta},~~\text{and } \curl V_\epsilon
\end{align*}
are relatively precompact in $W^{-1,s}(\Omega)$ for $s\in [1,\frac{3}{2})$.
The boundedness of $U_\epsilon$ and $V_\epsilon$ in $L^p((0,T)\times\Omega)$ for some $p>1$ can be shown as follows.

The sequence $\theta_\epsilon$ is bounded in $L^2((0,T)\times\Omega)$ and for the sequence $U_\epsilon$ we use the uniform estimates \eqref{eq 4.64}-\eqref{eq 4.68} and the special structure of $s(\rho_\epsilon,\theta_\epsilon)$ to conclude that it is bounded in $L^p((0,T)\times\Omega)$ for some $p>1$.
Moreover, $\epsilon (\log\theta_\epsilon-\log\rho_\epsilon-1)\to 0$ weakly in $L^p((0,T)\times\Omega)$.

Hence, we obtain that 
\begin{align}
\overline{s(\rho,\theta)\theta}=\overline{s(\rho,\theta)}\theta.
\end{align}
By the monotonicity of the entropy and the weak convergence we conclude that up to a subsequence 
\begin{align}
\theta_\epsilon\to \theta~~\text{a.a. in  } (0,T)\times\Omega.
\end{align}
In addition, we have that the limit temperature is positive a.a. on the set $(0,T)\times\Omega$, more precisely 
\begin{align}
\theta^{-3}\in L^1((0,T)\times\Omega).
\end{align}

Now, we can let $\epsilon\to 0$ in the approximate entropy equation.\\
Using the previous relations we obtain that
\begin{align*}
\frac{\kappa\nabla \theta_\epsilon}{\theta_\epsilon}\to \frac{\kappa\nabla \theta}{\theta}~~\text{weakly in } L^p(0,T;\Omega),
\end{align*}
for some $p>1$.\\
Applying the Div-Curl Lemma once more with $V_\epsilon=[u_\epsilon,0,0,0]$ and $U_\epsilon$ as before, we observe that
\begin{align*}
s(\rho_\epsilon,\theta_\epsilon)u_\epsilon\to \overline{s(\rho,\theta)}u~~\text{weakly in } L^p((0,T)\times\Omega)
\end{align*}
for some $p>1$.\\
The terms appearing in $\sigma_{\epsilon,\delta}$ are weakly lower semi-continuous as established in the previous part. 
Moreover, the $\epsilon$-dependent terms are non-negative.
Hence we can conclude that
\begin{align*}
\Sigma_{\epsilon,\delta}\to \sigma_\delta~~\text{weakly in } \mathcal{M}\big([0,T]\times\overline{\Omega}\big),
\end{align*}
where $\sigma_\delta$ is a positive measure on $[0,T]\times\overline{\Omega}$ satisfying 
\begin{align*}
\sigma_\delta\geq \frac{1}{\theta}\big(\mu|\nabla u|^2+\nu \rho|u|^2+\big(\frac{\kappa}{\theta}+\frac{\delta}{2}\theta^{\Gamma-2}\big)|\nabla \theta|^2\big)+\frac{\delta}{\theta^3}+\delta\theta^5.
\end{align*}

It remains to show that $\overline{s(\rho,\theta)}=s(\rho,\theta)$.
To this end, we have to show the strong convergence in the densities.
We follow the ideas presented in \cite[Chapter 3.6]{Novotny}. 

The first step is to introduce the test function $\varphi(x)$ for the momentum equation, where
\begin{align*}
\varphi(x)=\xi(x) \nabla \Delta^{-1} [\mathbb{1}_\Omega \rho_\epsilon] ~~\text{and  } \xi\in C_c^\infty(\Omega) .
\end{align*}
Taking $\varphi$ as an admissible test function in the Brinkman-type equation \eqref{eq 4.59} yields
\begin{align*}
\int_\Omega \rho\epsilon\theta_\epsilon \dv\varphi dx= \int_\Omega \mu \nabla u_\epsilon :\nabla \varphi +\nu \rho_\epsilon u_\epsilon \cdot\varphi dx
\end{align*}
Taking into account the specific form of the test function and integration by parts yields
\begin{align}\label{eq test func}
\int_\Omega \xi \bigg( \rho_\epsilon\theta_\epsilon \rho_\epsilon -\mu \nabla u_\epsilon : R[\mathbb{1}_\Omega \rho_\epsilon]-\nu \rho_\epsilon u_\epsilon \cdot \nabla \Delta^{-1} [\mathbb{1}_\Omega \rho_\epsilon]\bigg) dx= \sum_{i=1}^2 I_{i,\epsilon},
\end{align}
where
\begin{align*}
I_{1,\epsilon}&=-\int_\Omega \rho_\epsilon \theta_\epsilon \nabla \xi \cdot \nabla\Delta^{-1}[\mathbb{1}_\Omega \rho_\epsilon] dx\\
I_{2,\epsilon}&= \mu \int_\Omega \nabla u_\epsilon : \nabla \xi \otimes \nabla\Delta^{-1}[\mathbb{1}_\Omega \rho_\epsilon]  dx
\end{align*}
and the symbol $R$ denotes the double Riesz transform.\\
Repeating the same argument for the limit of the Brinkman equation \eqref{eq 4.73} with the test function $\varphi$ being 
\begin{align*}
\varphi(x)=\xi(x) \nabla \Delta^{-1} [\mathbb{1}_\Omega \rho]~~\text{and  } \xi\in C_c^\infty(\Omega) .
\end{align*}
yields the following
\begin{align}
\int_\Omega \xi \bigg( \rho\theta \rho -\mu \nabla u : R[\mathbb{1}_\Omega \rho]-\nu \rho u \cdot \nabla \Delta^{-1} [\mathbb{1}_\Omega \rho]\bigg) dx= \sum_{i=1}^2 I_{i},
\end{align}
where
\begin{align*}
I_{1}&=-\int_\Omega \rho\theta \nabla \xi \cdot \nabla\Delta^{-1}[\mathbb{1}_\Omega \rho] dx\\
I_{2}&= \mu \int_\Omega \nabla u : \nabla \xi \otimes \nabla\Delta^{-1}[\mathbb{1}_\Omega \rho]  dx.
\end{align*}
From the previous estimates we recall that 
\begin{align*}
\rho_\epsilon \to \rho ~~\text{ in } C_{weak}([0,T],L^2(\Omega)).
\end{align*}
Taking into account relations (4.64) - (4.77) we observe that the integral $I_{i,\epsilon}$ converges to its counterpart $I_i$ for $i=1,2$ and we infer
\begin{align}
\lim_{\epsilon\to 0} &\int_\Omega \xi \bigg( \rho_\epsilon\theta_\epsilon \rho_\epsilon -\mu \nabla u_\epsilon : R[\mathbb{1}_\Omega \rho_\epsilon]-\nu \rho_\epsilon u_\epsilon \cdot \nabla \Delta^{-1} [\mathbb{1}_\Omega \rho_\epsilon]\bigg) dx\\
&= \int_\Omega \xi \bigg( \rho\theta \rho -\mu \nabla u : R[\mathbb{1}_\Omega \rho]-\nu \rho u \cdot \nabla \Delta^{-1} [\mathbb{1}_\Omega \rho]\bigg) dx
\end{align}
Moreover, the last terms on the left-hand side of the equality converges to the last term on the right-hand side.
The next step is to rewrite the following term 
\begin{align*}
\int_\Omega \xi \mu \nabla u : R[\mathbb{1}_\Omega \rho]dx =\int_\Omega \mu R:[\xi \nabla u]\rho dx
\end{align*}  
where we used the properties of the double Riesz transform \cite[Chapter 11]{Novotny} and we observe that we can write
\begin{align*}
\mu R:[\xi \nabla u]= \mu \xi \dv u +\mu \omega(u),
\end{align*}
where $\omega(u)= R:[\xi \nabla u]- \xi R:[\nabla u]$ is the commutator.
Applying a result by Coifman and Meyer \cite{Coifman} and the previous bounds we obtain that
\begin{align}
\omega(u_\epsilon)\rho_\epsilon\to \overline{\omega(u)}\rho~~\text{weakly in } L^1((0,T)\times \Omega).
\end{align}
This yields
\begin{align*}
\overline{\omega(u)}=\omega(u).
\end{align*}
The proof of the convergence in (4.88) is shown by applying the Div-Curl Lemma to 
\[U_\epsilon=[\rho_\epsilon,\rho_\epsilon u_\epsilon]~~\text{and }V_\epsilon=[\omega(u_\epsilon),0,0,0]\]

Hence we obtain the following weak compactness identity for the effective pressure
\begin{align}
\overline{\rho \theta \rho}-\mu \overline{\rho \dv u}= \overline{\rho \theta}\rho -\mu \rho \dv u.
\end{align}
The final step is to multiply the continuity equation on $G'(\rho_\epsilon)$, with $G$ being a smooth and convex function. 
Then as $\epsilon \to 0$ we get
\begin{align*}
\int_\Omega \overline{G(\rho)}(t) dx+\int_0^t\int_\Omega \overline{\big(G'(\rho)\rho-G(\rho)\big)\dv u}dx dt\leq \int_\Omega G(\rho_0)dx 
\end{align*}
for all $t\in (0,T)$ and we deduce that
\begin{align*}
\int_\Omega  \overline{\rho \log \rho}(t) dx+\int_0^t\int_\Omega \overline{\rho \dv u}dx dt = \int_\Omega \rho_0 \log \rho_0 dx. 
\end{align*}
Via the theory of renormalized solutions by DiPerna and Lions \cite{DiPerna} we obtain 
\begin{align*}
\int_\Omega  \rho \log \rho(t) dx+\int_0^t\int_\Omega \rho \dv u dx dt \leq \int_\Omega \rho_0 \log \rho_0 dx. 
\end{align*}
Hence we obtain
\begin{align}
 \overline{\rho \theta \rho} \geq  \overline{\rho \theta}\rho
\end{align}
and as a consequence of of equation (4.89) 
\begin{align}
\overline{ \rho \dv u}\geq \rho \dv u.
\end{align}
Combining both estimates implies
\begin{align}
\overline{\rho \log \rho}=\rho \log \rho
\end{align}
which yields the desired strong convergence of the density as the function $x\log x$ is convex, i.e.
\begin{align}\label{eq strong rho}
\rho_\epsilon\to \rho ~~\text{ a.a. in } (0,T)\times \Omega.
\end{align}
This allows us to identify $\overline{s(\rho,\theta)}=s(\rho,\theta)$.

Having eliminated the $\epsilon$-dependent terms, we summarize the results.
For any $\delta>0$ we have constructed a trio $(\rho,u,\theta)$ solving the following equations.\\

The \textbf{continuity equation}
\begin{align}\label{eq 4.84}
\partial_t\rho + \dv(\rho u)=0~~\text{in } (0,T)\times\Omega,
\end{align}
with homogeneous Neumann boundary condition
\begin{align}
\nabla \rho \cdot n=0 ~~\text{on } (0,T)\times \partial\Omega,
\end{align}
and the initial condition
\begin{align}
\rho(0,\cdot)=\rho_{0,\delta}.
\end{align}

The \textbf{Brinkman-type equation}
\begin{align}
\int_\Omega \rho \theta \dv \phi dx= \int_\Omega \mu \nabla u :\nabla \phi dx+\int_\Omega \nu\rho u\cdot \phi dx,
\end{align}
for any test function $\phi \in C_c^\infty(\overline{\Omega})$ with
\begin{align}
\phi =0 ~~\text{on } \partial\Omega~~\text{ no-slip boundary conditions.}
\end{align}

The \textbf{approximate internal energy balance}
\begin{align}
\int_\Omega \rho\theta(t)d x=\int_\Omega \rho_{0,\delta}\theta_{0,\delta} d x+\int_0^t\int_\Omega \frac{\delta}{\theta^2}-\delta\theta^5\,d x d \tau.
\end{align}
for all $t\in [0,T]$.

The \textbf{approximate entropy inequality}
\begin{align}\begin{split}
&\int_0^T\int_\Omega s(\rho,\theta)\big(\partial_t\phi+u\cdot\nabla \phi\big)dxdt+\int_0^T\int_\Omega \frac{\nabla \theta}{\theta} \cdot \nabla \phi dxdt\\
&+\langle\sigma_\delta,\phi\rangle_{\mathcal{M}([0,T]\times\overline{\Omega}]}= -\int_\Omega s(\rho_{\delta,0},\theta_{\delta,0})\phi(0,\cdot)dx\end{split}
\end{align} 
for all $\phi\in C_c^\infty([0,T)\times\overline{\Omega})$ where
\begin{align}\label{eq 4.91}
\sigma_\delta\geq \frac{1}{\theta}\bigg(\mu|\nabla u|^2+\nu \rho|u|^2+\big(\frac{1}{\theta}+\frac{\delta}{2}\theta^{\Gamma-1}\big)|\nabla \theta|^2+\frac{\delta}{\theta^2}+\delta\theta^5\bigg).
\end{align}

\subsubsection{Limit $\delta\to 0$}
\label{d limit}
The last step in this proof is to let $\delta\to 0$.
To this end, let $(\rho_\delta,u_\delta,\theta_\delta)$ be a solution to the approximate system \eqref{eq 4.84}-\eqref{eq 4.91}.
We recall that the total mass of the system is conserved, i.e.
\begin{align}
\int_\Omega \rho_\delta(t,\cdot)dx =\int_\Omega \rho_{0,\delta}dx~~\text{for any } t\in [0,T].
\end{align}
We assume that 
\begin{align}
\rho_{0,\delta}\to \rho_0~~\text{in } L^1(\Omega),
\end{align}
and thus the bound is uniform for $\delta\to 0$.

Next, we apply the reverse Young's inequality to the energy balance equation
\begin{align}
\int_\Omega \frac{1}{S(\rho/\theta)^2}\frac{1}{4}(\rho^2+\theta^2)(t) dx \leq \int_\Omega \rho\theta(t) dx=\int_\Omega\rho_{0,\delta}\theta_{0,\delta}dx +\int_0^t\int_\Omega \frac{\delta}{\theta^2}-\delta\theta^5 dxd\tau,
\end{align}
where $S(h)=S(1/h)>0$ for $h>0$ is the Specht radius \cite{Tominaga}.
Since $S(h)\to \infty$ as $h\to 0$ we need the ratio $\rho/\theta$ to be bounded from below and above.
This follows from the fact that we chose positive initial data $\rho_0$ and $\theta_0$ and from the previous sections.
As a consequence the solution may exist only for a short time $T^*>0$.
 
Next, is the dissipation balance 
\begin{align}
\int_\Omega H_{\overline{\theta}}(\rho,\theta)(t) dx +\overline{\theta}\sigma_\delta\big[[0,t]\times\overline{\Omega}\big]=\int_\Omega  H_{\overline{\theta}}(\rho_{0,\delta},\theta_{0,\delta})dx +\int_0^t\int_\Omega\frac{\delta}{\theta^2}+\delta\theta^4 dxd\tau
\end{align}
satisfied for a.a. $t\in [0,T]$. 
Noting that the terms $\delta/\theta^2$ and $\delta\theta^4$ are absorbed in the entropy production $\sigma_\delta$ and the uniform bounds
\begin{align}
\int_\Omega  H_{\overline{\theta}}(\rho_{0,\delta},\theta_{0,\delta})dx\leq c~~\text{uniformly for } \delta\to 0
\end{align}
hold, we obtain the following uniform estimates depending only on the initial data:
\begin{align}\label{eq 4.97}
\essup_{t\in (0,T)}\|\rho_\delta(t)\|_{L^2(\Omega)}&\leq c,\\
\essup_{t\in (0,T)}\|\theta_\delta(t)\|_{L^2(\Omega)}&\leq c,\\\label{eq 4.99}
\sigma_\delta\big[[0,t]\times\overline{\Omega}\big]&\leq c.
\end{align}have that
\begin{align}
&\int_0^T\int_\Omega |\nabla \log \theta_\delta|^2 dxdt\leq c,\\
&\int_0^T\int_\Omega \frac{|\nabla u|^2+\rho |u|^2}{\theta_\delta}dxdt \leq c,\\
&\delta \int_0^T\int_\Omega \frac{1}{\theta_\delta^3}+\theta^5 dxdt\leq c,\\
&\delta \int_0^T\int_\Omega \theta^{\Gamma-2}|\nabla \theta_\delta|^2 dxdt\leq c.
\end{align}
Using that
\begin{align*}
\nabla \log \theta_\delta~~\text{bounded in }  L^2((0,T)\times \Omega)
\end{align*}
we obtain that
\begin{align}\label{eq 4.104}
\nabla \theta_\delta~~ \text{bounded in }  L^2((0,T)\times \Omega).
\end{align}
Moreover applying Lemma 3.1 and Proposition 3.2 we have that
\begin{align}
\log \theta_\delta~~\text{bounded in } L^2(0,T;H^1(\Omega)).
\end{align}
Combining the above estimates, we see that
\begin{align}
\theta_\delta ~~\text{is bounded in } L^2(0,T;H^1(\Omega))
\end{align}
and especially by the Sobolev embedding
\begin{align}\label{eq 4.107}
\theta_\delta ~~\text{is bounded in } L^2(0,T;L^6(\Omega)).
\end{align}
Next, we estimate
\begin{align}
\int_\Omega |\nabla u_\delta|^{3/2}=\int_\Omega\frac{|\nabla u_\delta|^{3/2} }{\theta_\delta^{3/4}}\theta_\delta^{3/4}\leq c\int_\Omega \frac{|\nabla u_\delta|^2}{\theta}+c\int_\Omega \theta^4_\delta\leq c.
\end{align}
Using the assumption that $\rho_0\in L^3(\Omega)$ and testing the continuity equation with $\rho^2$ we obtain 
\begin{align*}
&\int_\Omega \partial_t\rho \rho^2 dx +\int_\Omega \dv(\rho u)\rho^2dx =0.
\intertext{Integrating by parts yields}
&\frac{1}{3}\frac{d}{dt}\int_\Omega \rho^3 dx \leq \frac{2}{3}\int_\Omega\rho^3 |\dv u|\, dx.
\end{align*}
Thus, we obtain
\begin{align}
\frac{d}{dt}\|\rho_\delta\|_{L^3(\Omega)}^3\leq c\big( \|\rho_\delta\|_{L^3(\Omega)}^3\big)^3 +c\|\nabla u_\delta\|_{L^{3/2}(\Omega)}^{3/2},
\end{align}
where we used an inverse type of the Jensen inequality \cite{Takahasi}.
Note, this inequality requires lower and upper bounds on the density $\rho_\delta$ almost everywhere.
Thus, we have an ordinary differential equation of the type
\begin{align*}
x'\leq c_1 x^3 +C_2.
\end{align*}
This is an ODE with Lipschitz right-hand side and by Picarc-Lindeloeff Theorem the solution exists for a small time $T$. 
Hence,
\begin{align}
\rho_\delta~~\text{is bounded in } L^\infty(0,T;L^3(\Omega).
\end{align}
Now, we can conclude from the Brinkman equation that
\begin{align*}
\|\rho u^2\|_{L^1(\Omega)}+\|\nabla u_\delta\|^2_{L^2(\Omega)}\leq\int_\Omega \rho_\delta^2\theta^2_\delta \, dx \leq  c\|\rho_\delta\|_{L^3(\Omega)}^3+\|\theta_\delta\|^6_{L^6(\Omega)},
\end{align*}
where we used that $\rho$ is non-negative at least for a small time $T^*$.
And thus by the Poincare inequality
\begin{align}
u_\delta~~\text{is bounded in } H^1_0(\Omega).
\end{align}
We remark that by repeating the previous two steps we obtain that 
\begin{align}
\rho_\delta~~\text{is bounded in } L^\infty(0,T;L^p(\Omega)
\end{align}
for some $p>2$ at least for a small time $T>0$.\\

Now, we have all the necessary uniform estimates together in order to pass to the limit in the equations.

For the continuity equation we get
\begin{align}
\rho_\delta\to \rho ~~\text{weakly in } L^p((0,T)\times\Omega)
\end{align}
for some $p\geq 2$ and similarly
\begin{align}
\rho_\delta u_\delta\to \rho u ~~\text{weakly in } L^p((0,T)\times\Omega)
\end{align}
for some $p>1$.

After passing to the limit in the continuity equation as $\delta\to 0$ the limit satisfies the integral identity
\begin{align}
\int_0^T\int_\Omega\big(\rho \partial_t\phi + \rho u\cdot \nabla \phi\big) dxdt+\int_\Omega \rho_0dx =0
\end{align}
for any test function $\phi \in C_c^\infty([0,T)\times\overline{\Omega})$ i.e. $\rho,\, u$ satisfy in the following equation in the sense of distributions
\begin{align}
\partial_t\rho +\dv(\rho u)=0.
\end{align}
For the Stokes equation we have
\begin{align}
u_\delta\to u ~~\text{weakly in } H^1(\Omega)
\end{align}
and 
\begin{align}\label{eq 4.118}
\rho_\delta\theta_\delta\to \overline{\rho \theta}~~\text{weakly in } L^p((0,T)\times\Omega)
\end{align}
for some $p>1$.
Thus the limit satisfies
\begin{align}
\int_\Omega \nu \rho u\cdot\varphi +\int_\Omega \mu \nabla u :\nabla \varphi dx+\int_\Omega \rho \theta \dv\varphi dx =0,
\end{align}
for any $\varphi \in C_c^\infty(\overline{\Omega})$ with $\varphi|_{\partial\Omega}=0$ or in the sense of distributions
\begin{align}
-\nu\rho u +\mu \Delta u=\nabla(\rho \theta).
\end{align} 
Before passing to the limit in the approximate entropy equation we need to show the pointwise convergence of the temperature.
Again as in the previous sections the idea is to apply the Div-Curl Lemma.
Setting
\begin{align}
U_\delta=\bigg[s(\rho_\delta,\theta_\delta),s(\rho_\delta,\theta_\delta) u_\delta + \frac{\nabla \theta_\delta}{\theta_\delta}],~~V_\delta=[\theta_\delta,0,0,0]
\end{align}
we observe that
\begin{align*}
\dv U_\delta=\sigma_{\delta},~~\text{and } \curl V_\delta
\end{align*}
are relatively precompact in $W^{-1,s}(\Omega)$ for $s\in [1,\frac{3}{2})$.\\
Indeed,using equations \eqref{eq 4.104}-\eqref{eq 4.107} we see that 
\begin{align*}
\delta\theta^{\Gamma-1}_\delta\nabla \theta_\delta=\delta\frac{\Gamma}{2}\theta_\delta^{1/4}\theta_\delta^{\Gamma/2-1/4}\nabla \theta_\delta^{\Gamma/2}.
\end{align*}
Hence, we can conclude
\begin{align*}
\delta\theta^{\Gamma-1}_\delta\nabla \theta_\delta\to 0 ~~\text{in } L^p((0,T)\times \Omega) 
\end{align*}
as $\delta\to 0$ for a certain $p>1$.
In addition, since $\theta_\delta>0$ for a.a. $(t,x)$ we obtain
\begin{align*}
\delta\frac{1}{\theta_\delta^3}\to 0~~ \text{in } L^1((0,T)\times \Omega) 
\end{align*}
as $\delta\to 0$.\\

The boundedness of $U_\delta$ in $L^p((0,T)\times\Omega)$ for some $p>1$ can be shown as follows.\\
For the sequence $U_\delta$ we use the uniform estimates \eqref{eq 4.97}-\eqref{eq 4.99} and the special structure of $s(\rho_\delta,\theta_\delta)$ to conclude that it is bounded in $L^p((0,T)\times\Omega)$ for some $p>1$.
Hence, we obtain that 
\begin{align}
\overline{s(\rho,\theta)\theta}=\overline{s(\rho,\theta)}\theta.
\end{align}
By the monotonicity of the entropy and the weak convergence we conclude that up to a subsequence 
\begin{align}
\theta_\delta\to \theta~~\text{a.a. in }  (0,T)\times\Omega.
\end{align}
In addition, we have that the limit temperature is positive a.a. on the set $(0,T)\times\Omega$.

It remains to show that $\overline{s(\rho,\theta)}=s(\rho,\theta)$.
We proceed as in the previous section by showing the strong convergence of the density. 
From the bounds obtained in equations \eqref{eq 4.97}-\eqref{eq 4.107} we see that the methods from the $\epsilon$-limit can be applied in this setting too, cf. equations \eqref{eq test func}-\eqref{eq strong rho}.

Using the weak lower semi-continuity of convex functionals, we can let $\delta\to 0$ in the approximate entropy balance to conclude that
\begin{align}\begin{split}
&\int_0^T\int_\Omega \overline{s(\rho,\theta)}\big(\partial_t\phi +u\cdot \nabla \phi\big) dxdt-\int_0^T\int_\Omega \frac{\kappa \nabla \theta}{\theta}\cdot \nabla \phi dxdt\\
 &+\langle\sigma;\phi\rangle_{[\mathcal{M};C]([0,T]\times\overline{\Omega})}=-\int_\Omega s(\rho_0,\theta_0)dx	\end{split}
\end{align}
for any $\phi\in C_c^\infty([0,T]\times\overline{\Omega})$.
Here $\sigma\in \mathcal{M}^+([0,T]\times\overline{\Omega})$ is a weak-*-limit in the space of measures $ \mathcal{M}([0,T]\times\overline{\Omega})$ of the sequence $\sigma_\delta$.
Using the lower weak semi-continuity of convex functionals and the fact that all $\delta$-dependent terms in the entropy production rate are non-negative, we obtain that
\begin{align}
\sigma\geq \frac{1}{\theta}\big(\mu|\nabla u|^2+\nu \rho|u|^2+\frac{|\nabla \theta|^2}{\theta}\big).
\end{align}
The last step is to take the limit in the internal energy balance.
By equation \eqref{eq 4.118} we can pass to the limit and obtain
\begin{align}
\int_\Omega \rho\theta (t)dx =\int_\Omega \rho_0\theta_0 dx~~\text{for a.a. } t\in [0,T].
\end{align}

This completes the proof of the theorem.

\subsection{Higher regularity}
\label{higher reg}
In the proof of Theorem \ref{theorem 2}  we have noted that the weak solutions constructed by the approximate scheme satisfy better regularity and integrability properties.
\begin{theorem}[Regularity of weak solutions]\label{theorem 4}
Let $\Omega\subset	\mathbb{R}^n$, where $n=2,\,3$, be a bounded Lipschitz domain.
Assume that the initial data $\rho_0$, $E_0$ and $s_0$ satisfy the hypothesis of Theorem 4.1.\\
Then, in addition to the minimal regularity assumptions required in equations (3.2), (3.6), (3.10), (3.14)-(3.16), there holds:
\begin{enumerate}
\item[i)]The weak solution satisfies
\begin{align}
&\rho\in C_{weak}([0,T];L^3(\Omega))\cap C([0,T];L^1(\Omega)),\\
&u\in H^1_0(\Omega),\\
&\theta\in L^2(0,T;H^1(\Omega))\cap L^\infty(0,T:l^2(\Omega)),\\
&\log\theta\in L^2(0,T;H^1(\Omega)),
\end{align}
\item[ii)]The entropy satisfies
\begin{align}
ess \lim_{t\to 0^{+}}\int_\Omega s(\rho,\theta)(t,\cdot)\phi \,d x \geq \int_\Omega s(\rho_0,\theta_0)\phi \,d x~~\text{for any } \phi	\in C_c^\infty(\overline{\Omega}),~\phi\geq 0.
\end{align}
If in addition, $\theta_0\in W^{1,\infty}(\Omega)$ then
\begin{align}
ess \lim_{t\to 0^+}\int_\Omega s(\rho,\theta)(t,\cdot)\phi\,d x= \int_\Omega s(\rho_0,\theta_0)\phi \,d x~~\text{for all } \phi	\in C_c^\infty(\overline{\Omega}).
\end{align} 
\end{enumerate}
\end{theorem}
\begin{proof}
The integrability properties follow directly from the proof of the existence of weak solutions.

The strong continuity of the density is a general property of the transport equation in the context of renormalized solutions.

The last part of the proof follows step 3 in the proof of Theorem 3.2 in \cite{Novotny}.
\end{proof}

\section{Conclusion and Remarks}
\label{conclusion}
In this section we conclude with several remarks.\\

The first one is that in the first part of this paper we showed how to apply the energetic variational approach \cite{Lin} and \cite{Hyon} in the setting of fluid mechanics and combine it with the temperature in a natural way.
This leads to the general framework of the free energy as starting point for the thermodynamics of fluids.
With the choice of free energy (in terms of temperature and the
state variable, the phase field variable here), and the entropy production,
as well as the kinematics/transport of these variables, one should be
able to uniquely determine the system.

Second, we observe the importance of the Laplacian in the velocity term.
To this end , we recall the following estimates obtained from the continuity and momentum equation and the entropy production rate.
\begin{align}\label{eq 5.1}
\int_\Omega \nu\rho|u|^2 +\int_\Omega \mu |\nabla u|^2 &= \int_\Omega \rho \theta \dv u\\\label{eq 5.2}
\frac{1}{\gamma+1}\frac{d}{dt}\int_\Omega\rho^{\gamma+1}+\frac{\gamma}{\gamma+1}\int_\Omega\rho^{\gamma+1}\dv u&=0\\\label{eq 5.3}
\int_\Omega\frac{1}{\theta}\big(\mu |\nabla u|^2 +\nu \rho |u|^2\big) &\leq c
\end{align}
We note that equation \eqref{eq 5.3} gives a uniform bound on $\frac{\mu |\nabla u|^2}{\theta}$ and this bound is then used in equation \eqref{eq 5.1} and \eqref{eq 5.2} to obtain further estimates.
If we were to let the parameter $\mu$ go to $0$ we would loose these bounds, i.e. the control of the gradient of $u$ and thus we cannot achieve the ultimate goal to consider the ideal gas under a Darcy-type law.

Next, we note that adding a memory/ evolutionary term to the Brinkman-type equation, i.e. $(\rho u)_t$ does not change analysis of the model.
In addition, as noted earlier in Remark \ref{remark 9}, if we had a kinetic term in the total energy we would obtain the incompressible Navier-Stokes-Fourier system for which the existence of weak solutions for the ideal gas case is still open \cite{Novotny} and \cite{Sun}.
Thus our result gives in a sense a  "lower limit" existence result of the full compressible Navier-Stokes-Fourier system, where the difference and crucial aspect is the additional nonlinear term in the momentum equation.
 
Finally, we want to remark that for a similar system of equations with only Darcy-type dissipation we are able to show the well-posedness of the system in a critical Besov space \cite{Sulzbach21}. The difference in these two approaches is that the first one uses energy methods for finding the weak solution, whereas the second one uses scaling arguments and the algebra structure provided by the critical Besov space. 

\section*{Acknowledgments}
The authors would like to thank Prof. Anja Schl{\"o}merkemper for constructive suggestions and discussions.
This research was supported in part by the National Science Foundation Grant NSF DMS–1714401 and the United States–Israel Binational Science Foundation Grant BSF 2024246.  

\begin{appendix}
\section{Appendix}

The following result stating the weak convergence of a product of functions is due to \cite{Tartar}.
\begin{theorem}[Div-Curl Lemma]
Let $Q\subset \mathbb{R}^n$ be an open set. 
Assume 
\begin{align*}
U_n\to U ~~\text{weakly in } L^p(Q),\\
V_n\to V~~\text{weakly in } L^q(Q),
\end{align*}
where
\begin{align*}
\frac{1}{p}+\frac{1}{q}=\frac{1}{r}<1.
\end{align*}
In addition, let 
\begin{align*}
\dv U_n \equiv \nabla \cdot U_n~~\text{precompact in } W^{-1,s}(Q),\\
\curl V_n\equiv \nabla V_n-\nabla^T V_n~~\text{precompact in } W^{-1,s}(Q),
\end{align*}
for some $s>1$.
Then
\begin{align*}
U_n\cdot V_n\to U\cdot V~~\text{wekly in } L^r(Q).
\end{align*}
\end{theorem}

\begin{theorem}[Reverse Young's inequality \cite{Tominaga}]
Let $a$, $b$ be positive real numbers and let $\nu \in [0,1]$.
Then 
\begin{align*}
S\big(\frac{a}{b}\big)a^{1-\nu}b^\nu\geq (1-\nu)a+\nu b,
\end{align*}
where the constant $S(h)$ is called the Specht ratio and is defined by
\begin{align*}
S(h)=S(1/h)=\frac{h^{\frac{1}{h-1}}}{e\log h^{\frac{1}{h-1}}}~~h\neq 1,~\text{for } h>0.
\end{align*}
\end{theorem}
The next theorem stating the existence and uniqueness of ODEs is from \cite{Roubicek}.
\begin{theorem}[Carathéodory]
Let $T$ be fixed and $f:~I\times \mathbb{R}^n\to \mathbb{R}^n$ be a Carathéodory mapping satisfying the growth condition $|f(t,r)|\leq \gamma(t) + C|r|$ with some $\gamma\in L^1(I)$. Then:
\begin{itemize}
\item[i)] The initial-value problem 
\[\frac{du}{dt}= f(t,u(t))~\text{ for a.e. }~t\in I,~~u|_{t=0}= u_0\]
has a solution $u\in W^{1,1}(I;\mathbb{R}^n)$ on the interval $I=[0,T]$.
\item[ii)] If $f(t,\cdot)$ is also Lipschitz continuous in the sense $|f(t,r_1)-f(t,r_2)|\leq l(t)|r_1-r_2|$ with some $l\in L^1(I)$, then the solution is unique.
\end{itemize}
\end{theorem}
The next theorem is an inverse of the Jensen's inequality \cite{Takahasi}.
\begin{theorem}[Inverse Jensen's inequality]
Let $f$ be a measurable function on a probability measure space $(\Omega,\mathcal{F},\mu,)$ with $f(\Omega)\subset [m,M]$.
Then for $p \geq 1$ 
\begin{align*}
\int f^p d\mu \leq \alpha \bigg(\int fd\mu \bigg)^p+\beta
\end{align*}
holds for some $\alpha>0$ and $\beta=a(1-\frac{1}{p})x_0+b$ where
\begin{align*}
a=\frac{M^p-m^p}{M-m},~~b=\frac{Mm^p-mM^p}{M-m},~~m<x_0<M.
\end{align*}
\end{theorem}
The following two important theorems can be found in \cite{Novotny}.
\begin{theorem}
Let $I\subset \mathbb{R}$ be an interval, $Q\subset \mathbb{R}^n$ a domain and $(P,G)\in C(I)\times C(I)$ a couple of non-decreasing functions.
Assume that $\rho_n\in L^1(Q;I)$ is a sequence such that
\begin{align*}
\begin{rcases} P(\rho_n)\to \overline{P(\rho)}&\\
G(\rho_n)\to \overline{G(\rho)}&\\
P(\rho_n)G(\rho_n)\to \overline{P(\rho)G(\rho)}&
\end{rcases}~~ \text{weakly in } L^1(Q).
\end{align*}
\begin{itemize}
\item[i)]Then
\begin{align*}
\overline{P(\rho)} \overline{G(\rho)}\leq \overline{P(\rho)G(\rho)}.
\end{align*}
\item[ii)]If, in addition, $G\in C(\mathbb{R})$, $G(\mathbb{R})=\mathbb{R}$, $G$ strictly increasing and $P\in C(\mathbb{R})$, $P$ is non-decreasing then
\begin{align*}
\overline{P(\rho)}=P\circ G^{-1}\big(\overline{G(\rho)}\big).
\end{align*}
\item[iii)]In particular, if $G(z)=z$, then
\begin{align*}
\overline{P(\rho)}=P(\rho).
\end{align*}
\end{itemize}
\end{theorem}
\begin{theorem}
Let $Q\subset\mathbb{R}^n$ be a measurable set and $\{v_n\}_n$ a sequence of functions in $L^1(Q)$ such that
\begin{align*}
v_n\to v ~~\text{weakly in } L^1(Q).
\end{align*}
Let $\Phi:\mathbb{R}^m\to(-\infty,\infty]$ be a lower semi-continuous convex function.
Then
\begin{align*}
\int_Q \Phi(v) dx\leq \liminf_{n\to \infty}\int_Q\Phi(v_n)dx.
\end{align*}
Moreover, if
\begin{align*}
\Phi(v_n)\to \overline{\Phi(v)}~~\text{weakly in } L^1(Q),
\end{align*}
then
\begin{align*}
\Phi(v)\leq \overline{\Phi(v)}~~\text{a.a. on } Q.
\end{align*}
If, in addition, $\Phi$ is a strictly convex on an open convex set $U\subset \mathbb{R}^m)$, and
\begin{align*}
\Phi(v)= \overline{\Phi(v)}~~\text{a.a. on } Q,
\end{align*}
then
\begin{align*}
v_n(y)\to v(y)~~\text{for a.a. } y\in \{y\in Q|v(y)\in U\}
\end{align*}
extracting a subsequence as the case may be.
\end{theorem}
\end{appendix}

\printbibliography
\end{document}